\documentclass[11pt,reqno,a4paper]{amsart}
\usepackage{a4wide}
\usepackage[T1]{fontenc}
\usepackage[utf8]{inputenc}
\usepackage{lmodern}
\usepackage{microtype}
\usepackage{amsmath,amssymb,mathtools}
\usepackage[mathcal]{eucal}
\usepackage{enumitem}
\usepackage{url}
\usepackage[colorlinks=true,linkcolor=blue,citecolor=blue,urlcolor=blue]{hyperref}
\hypersetup{
  pdftitle={Twisting exponential spectra},
  pdfauthor={Bence Horv\'ath and Tomasz Kania}
}

\title[Twisting exponential spectra]{Twisting exponential spectra}

\author[B. Horv\'ath]{Bence Horv\'ath}
\address[B.~Horv\'ath]{Helvetia Insurance Ltd., Aeschengraben~21, Postfach, 4001 Basel, Switzerland}
\email{hotvath@gmail.com}

\author[T. Kania]{Tomasz Kania}
\address[T.~Kania]{Mathematical Institute\\Czech Academy of Sciences\\\v Zitn\'a 25 \\115 67 Praha 1\\Czech Republic  and  Institute of Mathematics and Computer Science\\ Jagiellonian University\\ {\L}ojasiewicza 6, 30-348 Krak\'{o}w, Poland
}
\email{kania@math.cas.cz, tomasz.marcin.kania@gmail.com}
\thanks{RVO: 67985840.}

\newcommand{\B}{\mathcal B}
\newcommand{\K}{\mathcal K}
\newcommand{\E}{\mathcal E}
\newcommand{\Calk}{\mathcal Q}
\newcommand{\C}{\mathbb C}
\newcommand{\R}{\mathbb R}
\newcommand{\N}{\mathbb N}
\newcommand{\Z}{\mathbb Z}
\newcommand{\Q}{\mathbb Q}
\newcommand{\Cq}{\mathbb C_{\mathbb Q}}
\newcommand{\Inv}{\operatorname{Inv}}
\newcommand{\Invzero}{\operatorname{Inv}_0}
\newcommand{\Exp}{\operatorname{Exp}}
\newcommand{\Lip}{\operatorname{Lip}}
\newcommand{\GL}{\operatorname{GL}}
\newcommand{\ind}{\operatorname{ind}}
\newcommand{\Ran}{\operatorname{Ran}}
\newcommand{\Ker}{\operatorname{Ker}}
\newcommand{\Lg}{\operatorname{Lg}}
\newcommand{\Rg}{\operatorname{Rg}}
\newcommand{\ltsr}{\operatorname{ltsr}}
\newcommand{\rtsr}{\operatorname{rtsr}}
\newcommand{\tsr}{\operatorname{tsr}}
\newcommand{\q}{\mathfrak q}
\newcommand{\IndG}{\mathfrak I}
\newcommand{\Jac}{\mathcal J}

\theoremstyle{plain}
\newtheorem{theorem}{Theorem}[section]
\newtheorem{proposition}[theorem]{Proposition}
\newtheorem{lemma}[theorem]{Lemma}
\newtheorem{corollary}[theorem]{Corollary}
\newtheorem*{theoremA}{Theorem A}
\newtheorem*{theoremB}{Theorem B}

\theoremstyle{definition}
\newtheorem{example}[theorem]{Example}

\theoremstyle{remark}
\newtheorem{remark}[theorem]{Remark}
\newtheorem*{monotonicityremark}{Monotonicity Remark}
\newtheorem{question}[theorem]{Question}

\numberwithin{equation}{section}

\begin{document}

\begin{abstract}
Klaja and Ransford exhibited a complex unital Banach algebra in which the
exponential spectra of \(ab\) and \(ba\) differ away from zero, and asked
whether this can occur in an algebra of bounded operators.  We answer their
question affirmatively.  For a Bourgain--Delbaen space \(X\) obtained from
Motakis' construction with Calkin algebra isomorphic as a Banach algebra to
\(C(S^4)\), we
construct \(S,T\in\B(X\oplus X)\) such that
\[
 \frac12\in\varepsilon_{\B(X\oplus X)}(ST)
 \quad\text{and}\quad
 \frac12\notin\varepsilon_{\B(X\oplus X)}(TS).
\]
Moreover,
\(\ltsr\B(X\oplus X)=\rtsr\B(X\oplus X)=2\), which is the least possible
stable rank for such an example.  More generally, if \(K\) is compact
metrisable and \(X\) is any space arising from Motakis' construction for
\(C(K)\), then, for every \(n\geqslant1\),
\[
 \ltsr\B(X^n)=\rtsr\B(X^n)=
 \begin{cases}
 \left\lceil \lfloor\dim K/2\rfloor/n\right\rceil+1,&\dim K<\infty,\\
 \infty,&\dim K=\infty.
 \end{cases}
\]
This also gives counterexamples of every finite stable rank at least two and
of infinite stable rank.  Finally, we relate every failure of
exponential-spectral commutativity to the kernel of the first matrix
stabilisation map on the index group.
\end{abstract}

\maketitle

\section{Introduction}

Let \(A\) be a complex unital Banach algebra.  We write \(\Inv(A)\) for the
group of invertible elements and \(\Invzero(A)\) for the connected component of
\(1_A\) in \(\Inv(A)\).  Equivalently,
\[
   \Invzero(A)=\Exp(A)=\{e^{a_1}\cdots e^{a_m}:a_1,\ldots,a_m\in A,\ m\geqslant1\}.
\]
The exponential spectrum of \(a\in A\), introduced by Harte \cite{Harte}, is
\[
   \varepsilon_A(a)=\{\lambda\in\C:\lambda1_A-a\notin\Invzero(A)\}.
\]
Harte's boundary theorem gives
\[
 \partial\varepsilon_A(a)\subseteq\partial\sigma_A(a)
 \subseteq\sigma_A(a)\subseteq\varepsilon_A(a),
\]
so the exponential spectrum is obtained from the ordinary spectrum by filling
some, possibly none, of its bounded complementary components
\cite[Theorem~1]{Harte}.

The ordinary spectrum satisfies Jacobson's commutativity law
\[
   \sigma_A(ab)\setminus\{0\}=\sigma_A(ba)\setminus\{0\}\qquad(a,b\in A),
\]
because \(1-ab\) is invertible if and only if \(1-ba\) is invertible.  Murphy
asked whether the analogous identity holds for the exponential spectrum
\cite[p.~237, immediately before Proposition~4.3]{Murphy}.  Klaja and
Ransford answered this in the negative by
constructing \(a,b\) in the Banach algebra
\(A_0=C(S^4,M_2(\C))\) for which
\[
   \varepsilon_{A_0}(ab)\setminus\{0\}
   \neq\varepsilon_{A_0}(ba)\setminus\{0\}.
\]
Their obstruction is the non-zero element of
\(\pi_4(\GL_2(\C))\cong\Z/2\Z\), represented by the suspension of the Hopf
map \cite[Theorems~1.2 and~3.1]{KlajaRansford}.  Polar decomposition gives a
deformation retraction of \(\GL_m(\C)\) onto \(U(m)\).  Since
\(U(m+1)/U(m)\cong S^{2m+1}\), the homotopy exact sequence shows that
\[
 \pi_j(U(m))\longrightarrow\pi_j(U(m+1))
\]
is an isomorphism for \(j<2m\).  Hence \(\pi_4(U(3))\) is stable, and
Bott periodicity gives
\[
 \pi_4(\GL_3(\C))\cong\pi_4(U)=0
\]
\cite[Theorem, p.~933]{BottStableHomotopy}.  Thus the class dies under the
next matrix stabilisation.  Klaja and Ransford then asked whether such a
failure can occur in \(\B(E)\) for some Banach space \(E\) \cite[Question~4.2]{KlajaRansford}.
Daniel and Ghosh proved the following transfer principle for a complex Banach
space \(Y\): if
\[
 \varepsilon_{\Calk(Y)}(uv)\setminus\{0\}
 =\varepsilon_{\Calk(Y)}(vu)\setminus\{0\}
 \qquad(u,v\in\Calk(Y)),
\]
then the corresponding identity holds for every pair in \(\B(Y)\)
\cite[Theorem~4.6]{DanielGhosh}.  This is a statement about the
\emph{exponential-spectral identity}, not merely about algebraic
commutativity.  In particular, \(\Calk(X)\cong C(S^4)\) is commutative as a
Banach algebra and so the theorem applies to \(\B(X)\).  Our operator space,
however, is \(E=X\oplus X\), and
\[
       \Calk(E)\cong M_2(\Calk(X))
       \cong M_2(C(S^4))\cong C(S^4,M_2(\C)),
\]
where the Klaja--Ransford pair shows that the hypothesis of their theorem
fails.  Thus the transfer principle is fully consistent with our example.

The construction belongs to the programme of prescribing Calkin algebras by
building Banach spaces with few operators.  Argyros and Haydon produced a
Bourgain--Delbaen \(\mathcal L_\infty\)-space with the scalar-plus-compact
property \cite{ArgyrosHaydon}.  Tarbard subsequently obtained
Bourgain--Delbaen spaces with finite-dimensional Calkin algebras and one whose
Calkin algebra is isomorphic to the convolution algebra
\(\ell_1(\mathbb N_0)\) \cite{Tarbard}.  Motakis, Puglisi and Zisimopoulou
then realised \(C(K)\) for every countable compact metrisable \(K\)
\cite{MotakisPuglisiZisimopoulou}.  Motakis later proved that every separable
commutative unital \(C^*\)-algebra occurs: for each compact metrisable \(K\),
his construction gives a Bourgain--Delbaen space whose Calkin algebra is,
after an equivalent renorming, isometrically isomorphic as a Banach algebra
to \(C(K)\) \cite{Motakis}.

For comparison with a classical space, if \(H\) is an infinite-dimensional
Hilbert space, then \(\B(H)\) contains two isometries with orthogonal ranges.
Rieffel's criterion therefore gives \(\tsr\B(H)=\infty\)
\cite[Proposition~6.5]{RieffelStableRank}; the continuous involution then
identifies the left and right ranks
\cite[Proposition~1.6]{RieffelStableRank}.

The compatible metric used in Motakis' construction is part of its auxiliary
data.  We verify in Section~4 that these data can be chosen so that any
prescribed countable family in \(C(K)\) is Lipschitz; this is the only addition
to Motakis' construction.  Choosing the entries of the Klaja--Ransford
matrices yields the desired lift to an operator algebra.  We write \(\ltsr\)
and \(\rtsr\) for the left and right topological stable ranks, and
\(\dim\) for covering dimension.

\begin{theoremA}
There are a Motakis realisation \(X\) over \(S^4\) and operators
\(S,T\in\B(E)\), where \(E=X\oplus X\), such that
\[
       \ltsr\B(E)=\rtsr\B(E)=2
       \quad\text{and}\quad
       \varepsilon_{\B(E)}(ST)\setminus\{0\}
       \neq
       \varepsilon_{\B(E)}(TS)\setminus\{0\}.
\]
More precisely,
\[
       \frac12\in\varepsilon_{\B(E)}(ST)
       \quad\text{and}\quad
       \frac12\notin\varepsilon_{\B(E)}(TS).
\]
\end{theoremA}

Stable rank two is the first possible value here.  Indeed, the left or right
stable-rank-one condition is equivalent to density of the invertibles, and
Murphy's argument then gives exponential-spectral commutativity.  Our second
main result applies uniformly to every Motakis realisation in the sense of
Section~4.

\begin{theoremB}
Let \(K\) be compact metrisable, let \(X\) be any Motakis realisation over
\(K\), and let \(n\geqslant1\).  Then
\[
   \ltsr\B(X^n)=\rtsr\B(X^n)
   =
   \begin{cases}
      \displaystyle
      \left\lceil\dfrac{\lfloor\dim K/2\rfloor}{n}\right\rceil+1,
        & \dim K<\infty,\\[3mm]
      \infty, & \dim K=\infty.
   \end{cases}
\]
In particular,
\[
   \ltsr\B(X)=\rtsr\B(X)=\tsr C(K),
\]
and \(\B(X)\) has stable rank one if and only if \(\dim K\leqslant1\).
\end{theoremB}

The proof of Theorem~B is an \(m\)-tuple Fredholm lifting argument.  The lower
bound comes from the Calkin quotient.  For the upper bound we approximate a
tuple by one which is unimodular modulo the compact operators and choose
diagonal representatives for the quotient coefficients.  The kernel of a
lower semi-Fredholm diagonal row is large enough to absorb its cokernel; a
finite-rank perturbation then produces an actually unimodular tuple.  A short
disjoint-union argument subsequently gives, as a corollary, counterexamples
of every stable rank in \(\{2,3,\ldots\}\cup\{\infty\}\).

We also record an elementary matrix factorisation.  It shows that every
failure of exponential-spectral commutativity yields a non-trivial element in
the kernel of the first matrix-stabilisation map on
\(\Inv(A)/\Invzero(A)\), thereby isolating the remaining torsion question of
Klaja and Ransford.
Finally, we calculate one particular Fredholm-index homomorphism on essentially
incomparable direct sums and show that it is insensitive to Jacobson twists.

\section{Elementary facts about identity components}

We begin with several standard facts.  They are included both to fix notation
and to isolate the component arguments used later.

\begin{lemma}\label{lem:identity_component_exponentials}
For every complex unital Banach algebra \(A\),
\[
 \Invzero(A)=\{e^{a_1}\cdots e^{a_m}:
                 a_1,\ldots,a_m\in A,\ m\geqslant1\}.
\]
\end{lemma}

\begin{proof}
Every product on the right is joined to \(1_A\) by multiplying the paths
\(t\mapsto e^{ta_j}\).  Conversely, \(\Inv(A)\) is open in \(A\) and
therefore locally path-connected, so every element of \(\Invzero(A)\) is
joined to \(1_A\) by a path \(u:[0,1]\to\Inv(A)\).  A sufficiently fine partition
\(0=t_0<\cdots<t_m=1\) makes
\(u(t_{j-1})^{-1}u(t_j)\) close enough to \(1_A\) to have a logarithm,
defined by its norm-convergent power series.  Multiplying these increments
writes \(u(1)\) as a product of exponentials.
\end{proof}

\begin{lemma}\label{lem:homomorphism_components}
Let \(\theta:A\to B\) be a continuous unital homomorphism of complex unital
Banach algebras.
\begin{enumerate}[label=\textup{(\roman*)}]
\item \(\theta(\Invzero(A))\subseteq\Invzero(B)\).
\item If \(\theta\) is onto, then
      \(\theta(\Invzero(A))=\Invzero(B)\).
\end{enumerate}
Consequently, if \(\theta\) is onto and \(x\in A\) satisfies
\(\theta(x)\notin\Invzero(B)\), then \(x\notin\Invzero(A)\).
\end{lemma}

\begin{proof}
The first assertion follows either by applying \(\theta\) to a path from
\(1_A\) to an element of \(\Invzero(A)\), or from
\(\theta(e^a)=e^{\theta(a)}\).  For the converse inclusion in (ii), write an
arbitrary element of \(\Invzero(B)\) as
\(e^{b_1}\cdots e^{b_m}\).  Choose \(a_j\in A\) with
\(\theta(a_j)=b_j\).  Then
\[
 e^{b_1}\cdots e^{b_m}
   =\theta(e^{a_1}\cdots e^{a_m})\in\theta(\Invzero(A)).
\]
The final assertion is immediate from (i).
\end{proof}

\begin{lemma}\label{lem:function_components}
Let \(K\) be a compact Hausdorff space and let \(B\) be a complex unital
Banach algebra.  Then
\[
       \Inv(C(K,B))=C(K,\Inv(B)).
\]
Moreover, \(f\in\Invzero(C(K,B))\) if and only if the map
\(f:K\to\Inv(B)\) is homotopic, through maps into \(\Inv(B)\), to the
constant map \(1_B\).
\end{lemma}

\begin{proof}
Invertibility is pointwise, and continuity of the pointwise inverse follows
from continuity of inversion on \(\Inv(B)\).  Hence
\[
       \Inv(C(K,B))=C(K,\Inv(B)).
\]
The invertible group of every unital Banach algebra is open.  Thus this is
an open subset of the Banach space \(C(K,B)\), so it is locally
path-connected; its connected components are therefore its path components.

A norm-continuous path \(t\mapsto f_t\) in \(C(K,\Inv(B))\) gives the
homotopy \(H(t,x)=f_t(x)\).  Conversely, if
\(H:[0,1]\times K\to\Inv(B)\) is continuous, uniform continuity on the
compact space \([0,1]\times K\) shows that
\(t\mapsto H(t,\cdot)\) is continuous in the uniform norm.  This proves
the asserted equivalence.
\end{proof}

\begin{lemma}\label{lem:scalar_component}
For every \(\lambda\in\C\setminus\{0\}\) and every complex unital Banach
algebra \(A\), the scalar \(\lambda1_A\) belongs to \(\Invzero(A)\).
Consequently, for \(u\in\Inv(A)\),
\[
       u\in\Invzero(A)
       \quad\Longleftrightarrow\quad
       \lambda u\in\Invzero(A).
\]
\end{lemma}

\begin{proof}
Choose \(z\in\C\) with \(e^z=\lambda\).  Then
\(\lambda1_A=e^{z1_A}\).  The second assertion follows because
\(\Invzero(A)\) is a subgroup of \(\Inv(A)\).
\end{proof}

\begin{lemma}[Jacobson's inverse formula]\label{lem:Jacobson_inverse}
Let \(A\) be a unital algebra and let \(a,b\in A\).  Then \(1-ab\) is
invertible if and only if \(1-ba\) is invertible.  More precisely, if
\(u=(1-ab)^{-1}\), then
\[
       (1-ba)^{-1}=1+bua.
\]
\end{lemma}

\begin{proof}
A direct multiplication, using \(u(1-ab)=(1-ab)u=1\), gives
\[
 (1-ba)(1+bua)=1=(1+bua)(1-ba).
\]
The converse follows by interchanging \(a\) and \(b\).
\end{proof}

We shall apply Lemma~\ref{lem:function_components} with
\(B=M_2(\C)\).  Thus membership in
\(\Invzero(C(K,M_2(\C)))\) is exactly null-homotopy of the corresponding
map \(K\to\GL_2(\C)\).

\section{The Klaja--Ransford element on \texorpdfstring{\(S^4\)}{S4}}

We use the model
\[
  S^4=\{(z_0,z_1,t)\in\C^2\times\R:
          |z_0|^2+|z_1|^2+t^2=1\}.
\]
Define \(a,b\in C(S^4,M_2(\C))\) by
\begin{equation}\label{eq:KR_ab}
 a(z_0,z_1,t)=\frac{1}{1+it}
       \begin{pmatrix} z_0&0\\ z_1&0\end{pmatrix},
 \qquad
 b(z_0,z_1,t)=\frac{1}{1+it}
       \begin{pmatrix} \overline z_0&\overline z_1\\0&0\end{pmatrix}.
\end{equation}
Put
\[
       c=I_2-2ab,
       \qquad
       d=I_2-2ba.
\]
A direct calculation gives
\begin{equation}\label{eq:c_formula}
 c(z_0,z_1,t)=I_2-\frac{2}{(1+it)^2}
      \begin{pmatrix}
        |z_0|^2&z_0\overline z_1\\
        z_1\overline z_0&|z_1|^2
      \end{pmatrix},
\end{equation}
and
\begin{equation}\label{eq:d_formula}
 d(z_0,z_1,t)=
       \begin{pmatrix}\phi(t)&0\\0&1\end{pmatrix},
 \qquad
 \phi(t)=-\left(\frac{1-it}{1+it}\right)^2.
\end{equation}
The following is the topological core of the construction of Klaja and
Ransford \cite{KlajaRansford}.  We give the details needed later.

\begin{proposition}\label{prop:KR}
For the functions \(c,d\) above,
\[
        c\in\Inv(C(S^4,M_2(\C)))\setminus
             \Invzero(C(S^4,M_2(\C)))
\]
and
\[
        d\in\Invzero(C(S^4,M_2(\C))).
\]
In fact, \(d\) is a single exponential.
\end{proposition}

\begin{proof}
For \(-1\leqslant t\leqslant1\),
\[
 \phi(t)
 =-\left(\frac{1-it}{1+it}\right)^2
 =\exp\bigl(i(\pi-4\arctan t)\bigr).
\]
Consequently,
\begin{equation}\label{eq:d_single_exp}
\begin{aligned}
 d&=\exp G,\\
 G(z_0,z_1,t)&=
 \begin{pmatrix}i\theta(z_0,z_1,t)&0\\0&0\end{pmatrix},\\
 \theta(z_0,z_1,t)&=\pi-4\arctan t.
\end{aligned}
\end{equation}
Thus \(d\in\Invzero(C(S^4,M_2(\C)))\).  Since
\(d=I_2-(2b)a\) is invertible, Lemma~\ref{lem:Jacobson_inverse}, applied
pointwise or in the function algebra, shows that
\(c=I_2-a(2b)\) is invertible as well.

Let \(p:M_2(\C)\to\C^2\) be projection onto the second column.  Since
\(c(x)\) is invertible, \(pc(x)\neq0\) for every \(x\in S^4\), and hence
\[
       f=\frac{pc}{\|pc\|}:S^4\longrightarrow S^3
\]
is well defined and continuous.  Formula~\eqref{eq:c_formula} gives
\begin{equation}\label{eq:pc_formula}
 pc(z_0,z_1,t)=
 \left(
      -\frac{2z_0\overline z_1}{(1+it)^2},
      1-\frac{2|z_1|^2}{(1+it)^2}
 \right).
\end{equation}
The imaginary part of its second coordinate is
\begin{equation}\label{eq:pc_imaginary_part}
 \operatorname{Im}\left(1-\frac{2|z_1|^2}{(1+it)^2}\right)
       =\frac{4t|z_1|^2}{(1+t^2)^2}.
\end{equation}
It follows that \(f\) maps the closed upper, respectively lower, hemisphere
of \(S^4\) into the closed upper, respectively lower, hemisphere of \(S^3\).
On the equator \(t=0\), one has \(\|pc\|=1\), because
\[
 4|z_0|^2|z_1|^2+
       (|z_0|^2-|z_1|^2)^2=1,
\]
and therefore
\begin{equation}\label{eq:f_equator}
 f(z_0,z_1,0)
   =(-2z_0\overline z_1,|z_0|^2-|z_1|^2).
\end{equation}

The map in \eqref{eq:f_equator} is the Hopf map
\(h:S^3\to S^2\).  Its suspension \(Eh:S^4\to S^3\) is, away from the two
poles and with the continuous extension at them, given by
\begin{equation}\label{eq:suspended_Hopf}
 Eh(z_0,z_1,t)=
 \left(
       \frac{-2z_0\overline z_1}{\sqrt{1-t^2}},
       \frac{|z_0|^2-|z_1|^2}{\sqrt{1-t^2}}+it
 \right).
\end{equation}
It agrees with \(f\) on the equator and sends the open upper, respectively
lower, hemisphere of \(S^4\) into the open upper, respectively lower,
hemisphere of \(S^3\).  In view of \eqref{eq:pc_imaginary_part}, the points
\(f(x)\) and \(Eh(x)\) are never antipodal: off the equator, \(-Eh(x)\)
lies in the open hemisphere opposite to the closed hemisphere containing
\(f(x)\), while on the equator the two maps agree.  Hence
\[
 H_s(x)=\frac{(1-s)f(x)+sEh(x)}
              {\|(1-s)f(x)+sEh(x)\|}
       \qquad(0\leqslant s\leqslant1)
\]
defines a homotopy from \(f\) to \(Eh\).

The Hopf map represents a generator of \(\pi_3(S^2)\).  Freudenthal's
suspension theorem says that
\[
 E:\pi_3(S^2)\longrightarrow\pi_4(S^3)
\]
is surjective; hence the suspension of a Hopf generator represents the
non-zero element of \(\pi_4(S^3)\cong\Z/2\Z\); see
\cite{Hopf,Freudenthal,Hu}.  Thus \(Eh\), and therefore \(f\), is not
null-homotopic.

Finally, suppose that \(c\) belonged to
\(\Invzero(C(S^4,M_2(\C)))\).  By
Lemma~\ref{lem:function_components}, there would be a homotopy
\(C_s:S^4\to\GL_2(\C)\) from \(c\) to \(I_2\).  Every second column
\(pC_s(x)\) is non-zero, so
\[
       (s,x)\longmapsto \frac{pC_s(x)}{\|pC_s(x)\|}
\]
would be a null-homotopy of \(f\), a contradiction.
\end{proof}

\begin{remark}\label{rem:spectra_disk_circle}
In this example the ordinary spectra of \(ab\) and \(ba\) are both the circle
with centre \(1/2\) and radius \(1/2\).  The exponential spectrum of \(ba\)
is that circle, whereas the exponential spectrum of \(ab\) is the closed disc
bounded by it \cite{KlajaRansford}.  Thus the failure is a filling-in
phenomenon, not a failure of the ordinary Jacobson identity.
\end{remark}

\section{Motakis realisations with prescribed diagonal lifts}

For a Banach space \(X\), write
\[
        \Calk(X)=\B(X)/\K(X)
\]
for its Calkin algebra and let
\(\q_X:\B(X)\to\Calk(X)\) be the quotient map.

Motakis' construction replaces a compatible metric on the compactum by an
equivalent one built from a countable uniformly dense \(\Cq\)-linear subspace
of \(C(K)\), where \(\Cq=\Q+i\Q\).  The next formulation incorporates a
prescribed countable family and makes its members Lipschitz for the new metric,
as needed for the multiplicative lifting below.

If \(\Gamma\) is a countable index set and
\((d_\gamma)_{\gamma\in\Gamma}\) is a Schauder basis with coordinate
functionals \((d_\gamma^*)_{\gamma\in\Gamma}\), an operator \(D\) is
\emph{diagonal with respect to \((d_\gamma)\)} if
\[
 d_\gamma^*(Dd_\eta)=0\qquad(\gamma\neq\eta),
\]
or, equivalently, if \(Dd_\gamma=\lambda_\gamma d_\gamma\) for suitable
scalars \(\lambda_\gamma\).

\begin{theorem}[Motakis; prescribed-family form]
\label{thm:Motakis}
Let \(K\) be a compact metrisable space and let
\(\mathcal P\subseteq C(K)\) be countable.  The auxiliary data in Motakis'
construction may be chosen so that the resulting compatible metric \(\rho\)
on \(K\) and the resulting Argyros--Haydon-type Bourgain--Delbaen
\(\mathcal L_\infty\)-space \(X\), with its Schauder basis
\((d_\gamma)_{\gamma\in\Gamma}\), where \(\Gamma\) is countable and the
basis is listed in the rank order of the construction,
have the following properties.
\begin{enumerate}[label=\textup{(\roman*)}]
\item Every member of \(\mathcal P\) is Lipschitz on \((K,\rho)\).
\item Every \(T\in\B(X)\) can be written
\[
       T=D+K_0,
\]
where \(D\) is diagonal with respect to \((d_\gamma)\) and
\(K_0\in\K(X)\).
\item There is a unital Banach-algebra isomorphism
\[
        \Psi:C(K)\longrightarrow\Calk(X).
\]
After an equivalent renorming of \(X\), this isomorphism is isometric.
\item There is a map \(\kappa:\Gamma\to K\) such that, for every
\(f\in\Lip(K,\rho)\), the diagonal prescription
\begin{equation}\label{eq:Motakis_diagonal}
       \widehat f(d_\gamma)=f(\kappa(\gamma))d_\gamma
       \qquad(\gamma\in\Gamma)
\end{equation}
defines an operator \(\widehat f\in\B(X)\).  The map
\[
       \Lip(K,\rho)\longrightarrow\B(X),
       \qquad f\longmapsto\widehat f,
\]
is a unital algebra homomorphism and
\[
       \q_X(\widehat f)=\Psi(f)
       \qquad(f\in\Lip(K,\rho)).
\]
\end{enumerate}
\end{theorem}

\begin{remark}[The auxiliary family in Motakis' construction]
\label{rem:Motakis_auxiliary_hypotheses}
Immediately before Instruction~3.6, Motakis imposes precisely the following
hypotheses:
\begin{equation*}\tag{M}
 \begin{gathered}
 \mathcal A\subseteq C(K)\text{ is countable and uniformly dense},\\
 \alpha f+\beta g\in\mathcal A
 \quad
 (f,g\in\mathcal A,\ \alpha,\beta\in\Cq).
 \end{gathered}
\end{equation*}
See \cite[Section~3.4, paragraph preceding Instruction~3.6]{Motakis}.
No multiplicative closure or invariance under complex conjugation is assumed;
the products defining \(\mathcal U_n\) are formed in \(C(K)\) and need not
belong to \(\mathcal A\).
\end{remark}

\begin{proof}
Since \(C(K)\) is separable, choose a countable uniformly dense set
\(\mathcal D\subseteq C(K)\), and put
\[
 \mathcal A=\operatorname{span}_{\Cq}
       \bigl(\mathcal D\cup\mathcal P\cup\{1_K\}\bigr).
\]
Thus \(\mathcal A\) is countable, uniformly dense in \(C(K)\), and closed
under \(\Cq\)-linear combinations, so it satisfies
condition~\textup{(M)}.  Enumerate
\[
       \mathcal U=\{f\in\mathcal A:\|f\|_\infty\leqslant1\}
       =\{\phi_1,\phi_2,\ldots\}.
\]
For \(n\in\N\), define, exactly as in \cite[Section~3.4]{Motakis},
\begin{equation}\label{eq:Motakis_Un}
 \mathcal U_n=
 \left\{
   \vartheta^k\phi_{i_1}\cdots\phi_{i_k}:
   k\geqslant1,\ 0\leqslant\vartheta\leqslant\frac{n}{n+1},\
   1\leqslant i_1,\ldots,i_k\leqslant n
 \right\}.
\end{equation}
Taking \(\vartheta=0\) shows explicitly that \(0\in\mathcal U_n\).  The
family is increasing and every \(\mathcal U_n\) is contained in the
closed unit ball of \(C(K)\).  For fixed \(n\) and \(k\), the corresponding
set in \eqref{eq:Motakis_Un} is a finite union of compact images of
\([0,n/(n+1)]\).  From any sequence in \(\mathcal U_n\), either the word
lengths have a bounded subsequence, in which case one passes further to a
fixed word and then to a convergent sequence of parameters, or they have a
subsequence tending to infinity, in which case the functions converge
uniformly to \(0\), since their norms are at most \((n/(n+1))^k\).  Thus
\(\mathcal U_n\) is compact.  It is a multiplicative semigroup: the product of words of lengths
\(k\) and \(\ell\), with parameters \(\vartheta\) and \(\eta\), is a word
of length \(k+\ell\) with parameter
\[
       \zeta=(\vartheta^k\eta^\ell)^{1/(k+\ell)}
       \leqslant \frac{n}{n+1}.
\]

Define
\begin{equation}\label{eq:Motakis_metric}
 \rho(x,y)=\sum_{n=1}^\infty 2^{-n}
       \max\{|g(x)-g(y)|:g\in\mathcal U_n\}.
\end{equation}
The evaluation map
\((x,g)\mapsto g(x)\) on \(K\times\mathcal U_n\) is continuous and its
domain is compact.  Equivalently, \(\mathcal U_n\) is equicontinuous by the
Arzel\`a--Ascoli theorem.  Uniform continuity then shows that
\[
 (x,y)\longmapsto\max_{g\in\mathcal U_n}|g(x)-g(y)|
\]
is continuous.  Each summand is a pseudometric, and the series
converges uniformly on \(K\times K\); hence \(\rho\) is a continuous
pseudometric.  It separates points.  Indeed, if \(x\neq y\), choose
\(h\in\mathcal A\) with \(h(x)\neq h(y)\), choose \(q\in\Q\) with
\(q>\|h\|_\infty\), and write \(h/q=\phi_j\).  For every \(n\geqslant j\),
the function \((n/(n+1))\phi_j\) belongs to \(\mathcal U_n\), so
\(\rho(x,y)>0\).  Thus \(\rho\) is a metric.
The identity map from the original compact space \(K\) to the Hausdorff
metric space \((K,\rho)\) is a continuous bijection, and hence a
homeomorphism.  Therefore \(\rho\) is compatible with the given topology.
Moreover,
\begin{equation}\label{eq:Motakis_Un_Lipschitz}
 |g(x)-g(y)|\leqslant 2^n\rho(x,y)
       \qquad(g\in\mathcal U_n).
\end{equation}
Together with \(\|g\|_\infty\leqslant1\), this is precisely the uniform
estimate used in \cite[proof of Proposition~4.4]{Motakis}, where
Proposition~4.1 there is applied with \(L=2^n\) and \(M=1\).  It is not an
additional hypothesis of Instruction~3.6.

Now let \(0\neq f\in\mathcal P\).  Choose
\(q\in\Q\), \(q>\|f\|_\infty\), and write \(f/q=\phi_j\).  For any
\(n\geqslant j\), the function
\((n/(n+1))\phi_j\) belongs to \(\mathcal U_n\), and hence
\[
 \rho(x,y)\geqslant
 2^{-n}\frac{n}{n+1}\frac{|f(x)-f(y)|}{q}.
\]
Thus
\[
 |f(x)-f(y)|\leqslant
 q2^n\frac{n+1}{n}\rho(x,y),
\]
so \(f\) is Lipschitz.  The zero function is Lipschitz trivially.

For completeness, we record the remaining auxiliary choices.  Motakis fixes
positive even integers \((m_j,n_j)_{j\geqslant1}\) satisfying
\begin{equation}\label{eq:Motakis_growth}
 \begin{gathered}
 m_1\geqslant8,\qquad m_{j+1}\geqslant m_j^2,\qquad
 n_1\geqslant m_1^2,\\
 n_{j+1}\geqslant(16n_j)^{\log_2m_{j+1}}
       \qquad(j\geqslant1),
 \end{gathered}
\end{equation}
as in \cite[Section~2.1]{Motakis}.  Compactness of \((K,\rho)\) allows one
to choose increasing finite sets \(K_n\subseteq K\) such that \(K_n\) is an
\(m_n^{-1}2^{-(n+1)}\)-net.  The preliminary index set
\(\overline\Gamma\) formed in Instruction~3.4 of \cite{Motakis} is
countable.  Consequently, the coding injection
\(\sigma:\overline\Gamma\to\N\) may be chosen to satisfy
\begin{equation}\label{eq:Motakis_coding}
       m_{4\sigma(\gamma)}>2^{\operatorname{rank}(\gamma)}
       \qquad(\gamma\in\overline\Gamma).
\end{equation}
Equations \eqref{eq:Motakis_Un}--\eqref{eq:Motakis_coding} give the
auxiliary choices relevant here.  With the remaining standard
Bourgain--Delbaen data from \cite[Instruction~3.6]{Motakis}, that
instruction and the subsequent arguments apply without change.

We now identify the quoted conclusions.  The Schauder basis is part of
\cite[Theorem~A]{Motakis}.  Proposition~4.1 of \cite{Motakis}
proves that \(\widehat f\) is bounded for every
\(f\in\Lip(K,\rho)\).  Section~4.2 defines
\(\Psi(f)=\q_X(\widehat f)\) on Lipschitz functions, and Proposition~4.3
extends \(\Psi\) to a Banach-algebra embedding of \(C(K)\).
Proposition~4.4 supplies the equivalent norm making this embedding
isometric.  Theorems~5.1--5.2 and Corollary~5.3 prove that it is onto,
giving (iii), while Corollary~5.4 is exactly (ii).  Finally, unitality and
multiplicativity in (iv) follow on the basis from
\[
 \widehat{fg}(d_\gamma)
 =f(\kappa(\gamma))g(\kappa(\gamma))d_\gamma
 =\widehat f\,\widehat g(d_\gamma),
 \qquad \widehat{1_K}=I_X.
\]
Together with the preceding argument for \textup{(i)}, this proves the theorem.
Only the choice making \(\mathcal P\) Lipschitz refines Motakis' construction;
all the operator-theoretic conclusions are his.
\end{proof}

We call any output obtained from admissible choices a \emph{Motakis
realisation over \(K\)}.  Such realisations need not be isomorphic: Motakis
obtains continuum many pairwise very incomparable ones by varying the
auxiliary sequences \cite[Proposition~8.7]{Motakis}.  All statements below are
uniform over the admissible choices.

\begin{remark}\label{rem:unbounded_diagonal_map}
The homomorphism \(f\mapsto\widehat f\) in
Theorem~\ref{thm:Motakis} need not be bounded when
\(\Lip(K,\rho)\) is equipped only with the uniform norm; Motakis explicitly
notes this point.  It is the Calkin-valued map
\(f\mapsto\q_X(\widehat f)\) which extends continuously to all of \(C(K)\).
Accordingly, identities involving an infinite power series of lifted
functions must be justified on the basis, rather than by applying continuity
of the lifting map.
\end{remark}

\begin{lemma}\label{lem:diagonal_representative}
Let \(X\) be a Motakis realisation over \(K\) as in
Theorem~\ref{thm:Motakis}.  Every element of \(\Calk(X)\) has a bounded
diagonal representative in \(\B(X)\).
\end{lemma}

\begin{proof}
Given \(\dot T\in\Calk(X)\), choose \(T\in\B(X)\) with
\(\q_X(T)=\dot T\).  Write \(T=D+K_0\) as in
Theorem~\ref{thm:Motakis}(ii).  Then \(D\) is bounded and diagonal, and
\(\q_X(D)=\dot T\).
\end{proof}

For a fixed finite direct-sum norm on \(X^n\), compactness of a matrix
operator is equivalent to compactness of all of its entries, and therefore
\[
       \Calk(X^n)\cong M_n(\Calk(X)).
\]

\begin{lemma}\label{lem:matrix_lift}
Let \(X\) be the Motakis realisation over \(K\) fixed above, let
\(n\in\N\), and let
\(\Phi=(\phi_{ij})\in M_n(\Lip(K,\rho))\).  Define
\[
       \widehat\Phi=(\widehat{\phi_{ij}})_{i,j=1}^n\in\B(X^n).
\]
Then, for \(\Phi,\Theta\in M_n(\Lip(K,\rho))\),
\begin{equation}\label{eq:matrix_lift_multiplicative}
       \widehat{\Phi\Theta}=\widehat\Phi\,\widehat\Theta,
       \qquad
       \widehat{I_n}=I_{X^n}.
\end{equation}
Under the natural identification
\(\Calk(X^n)\cong M_n(\Calk(X))\),
\begin{equation}\label{eq:matrix_lift_quotient}
       \q_{X^n}(\widehat\Phi)=(\Psi(\phi_{ij}))_{i,j=1}^n.
\end{equation}
Moreover, the entries of the pointwise matrix exponential \(e^\Phi\) are
Lipschitz and
\begin{equation}\label{eq:matrix_lift_exponential}
       e^{\widehat\Phi}=\widehat{e^\Phi}.
\end{equation}
\end{lemma}

\begin{proof}
The first two assertions follow entrywise from the multiplicativity and
unitality of \(f\mapsto\widehat f\).  The quotient formula follows in the
same way from \(\q_X(\widehat f)=\Psi(f)\).

Choose any norm on \(M_n(\C)\).  Since \(\Phi(K)\) is bounded and \(\Phi\) is
Lipschitz, the identity
\[
 e^P-e^Q=\int_0^1e^{(1-s)P}(P-Q)e^{sQ}\,ds
       \qquad(P,Q\in M_n(\C))
\]
shows that \(x\mapsto e^{\Phi(x)}\) is Lipschitz.  To prove
\eqref{eq:matrix_lift_exponential}, for each \(\gamma\in\Gamma\) consider
\[
 V_\gamma=\operatorname{span}\{(0,\ldots,0,d_\gamma,0,\ldots,0)\}
          \subseteq X^n,
\]
where the non-zero coordinate ranges over the \(n\) positions.  This is an
\(n\)-dimensional invariant subspace for \(\widehat\Phi\), and the matrix of
the restriction is \(\Phi(\kappa(\gamma))\).  Hence the restrictions of
\(e^{\widehat\Phi}\) and \(\widehat{e^\Phi}\) to every \(V_\gamma\) agree.
They therefore agree on the linear span of the basis of \(X^n\), and by
continuity on all of \(X^n\).
\end{proof}

\section{The operator-algebra counterexample}

We first prove the basic \(S^4\)-case.  The prescribed-stable-rank form of
the example will be obtained as a corollary of the computation in
Section~\ref{sec:stable_rank}.

\begin{theorem}\label{thm:main}
There are a Motakis realisation \(X\) over \(S^4\) and operators
\(S,T\in\B(E)\), where \(E=X^2\), such that
\[
       \frac12\in \varepsilon_{\B(E)}(ST)
       \quad\text{and}\quad
       \frac12\notin \varepsilon_{\B(E)}(TS).
\]
Consequently,
\[
       \varepsilon_{\B(E)}(ST)\setminus\{0\}
       \neq
       \varepsilon_{\B(E)}(TS)\setminus\{0\}.
\]
\end{theorem}

\begin{proof}
Let \(a,b,c,d\) and \(\theta\) be as in Section~3.  Apply
Theorem~\ref{thm:Motakis} to \(K=S^4\) and to the finite family consisting
of all scalar entries of \(a\) and \(b\), together with \(\theta\).  For the
resulting compatible metric these functions are Lipschitz.  Let \(X\) be
the resulting Motakis realisation over \(S^4\), put \(E=X^2\), and define
\[
       S=\widehat a,
       \qquad
       T=\widehat b
       \qquad\text{in }\B(E).
\]
By Lemma~\ref{lem:matrix_lift},
\begin{equation}\label{eq:lifts_c_d}
       I_E-2ST=\widehat c,
       \qquad
       I_E-2TS=\widehat d.
\end{equation}

Suppose that \(I_E-2ST\in\Invzero(\B(E))\).  Its Calkin image would then
belong to \(\Invzero(\Calk(E))\) by
Lemma~\ref{lem:homomorphism_components}.  Under the Banach-algebra
isomorphisms
\[
 \Calk(E)\cong M_2(\Calk(X))
       \cong M_2(C(S^4))
       \cong C(S^4,M_2(\C)),
\]
formula~\eqref{eq:matrix_lift_quotient} identifies this Calkin image with
\(c\).  This contradicts Proposition~\ref{prop:KR}.  Hence
\begin{equation}\label{eq:not_exp_lift_c}
       I_E-2ST\notin\Invzero(\B(E)).
\end{equation}

Let \(G\) be the matrix function in \eqref{eq:d_single_exp}.  Its entries
are Lipschitz, and Lemma~\ref{lem:matrix_lift} gives
\[
       I_E-2TS=\widehat d
       =\widehat{e^G}=e^{\widehat G}
       \in\Invzero(\B(E)).
\]
In particular, \(I_E-2TS\) is invertible.  Jacobson's inverse formula now
also shows that \(I_E-2ST\) is invertible.  Thus the obstruction in
\eqref{eq:not_exp_lift_c} is a component obstruction.

Finally, \((1/2)I_E\in\Invzero(\B(E))\) by
Lemma~\ref{lem:scalar_component}.  Therefore
\[
 \frac12I_E-ST=\frac12(I_E-2ST)\notin\Invzero(\B(E)),
\]
whereas
\[
 \frac12I_E-TS=\frac12(I_E-2TS)\in\Invzero(\B(E)).
\]
This is the asserted spectral separation.
\end{proof}

\section{Stable rank of the Motakis operator algebras}\label{sec:stable_rank}

We now compute the left and right topological stable ranks of the operator
algebras used above.  For a unital Banach algebra \(A\) and \(m\in\N\), put
\[
  \Lg_m(A)=\{(a_1,\ldots,a_m)\in A^m:
        Aa_1+\cdots+Aa_m=A\}
\]
and
\[
  \Rg_m(A)=\{(a_1,\ldots,a_m)\in A^m:
        a_1A+\cdots+a_mA=A\}.
\]
The left, respectively right, topological stable rank is the least \(m\) for
which \(\Lg_m(A)\), respectively \(\Rg_m(A)\), is dense in \(A^m\), with
value \(\infty\) if no such \(m\) exists
\cite[Definition~1.4]{RieffelStableRank}.  For a unital \(C^*\)-algebra the
involution interchanges them, and we write \(\tsr\) for their common value
\cite[Proposition~1.6]{RieffelStableRank}.  His
Theorem~6.1, however, begins verbatim,
``Let \(A\) be a \(C^*\)-algebra.  Then for any positive integer \(m\),''
and so cannot itself be applied to \(\B(X)\)
\cite[Theorem~6.1]{RieffelStableRank}.  The following lemma is the
Banach-algebraic form of the rectangular-matrix argument in Rieffel's
Definition~6.2, Lemma~6.3, and the upper-bound part of his Theorem~6.1
\cite[Definition~6.2, Lemma~6.3, and pp.~315--316]{RieffelStableRank}.
Rieffel states these results in the \(C^*\)-algebra setting, but this part
of the proof uses only the structure of a unital Banach algebra.  Corach
and Larotonda studied the Banach-manifold and fibration properties of the
corresponding spaces of rectangular unimodular matrices
\cite[Proposition~1.3 and Theorem~1.6]{CorachLarotondaUnimodular}, but we
have not found the density assertion below in their work.  Nica later
recorded the full matrix formula in Banach-algebra generality, using the
left topological stable rank and citing Rieffel
\cite[Remark~3.2(g) and Theorem~7.1]{NicaStableRanks}.  We include the short
argument to make its scope and our one-sided conventions explicit.

\begin{monotonicityremark}
For \(1\leqslant m\leqslant k\),
\[
 \Lg_m(A)\times A^{k-m}\subseteq\Lg_k(A),
 \qquad
 \Rg_m(A)\times A^{k-m}\subseteq\Rg_k(A).
\]
Consequently, density of \(\Lg_m(A)\), respectively of \(\Rg_m(A)\),
implies density at every level \(k\geqslant m\): perturb the first \(m\)
entries and leave the remaining entries fixed, padding the coefficients of
a unimodularity identity by zeros.
\end{monotonicityremark}

\begin{lemma}\label{lem:Banach_matrix_upper_bound}
Let \(A\) be a complex unital Banach algebra with
\(d=\ltsr(A)<\infty\), and let \(p\geqslant q\geqslant1\).
If
\[
          p-q+1\geqslant d,
\]
then the left-invertible elements of \(M_{p,q}(A)\) are dense in
\(M_{p,q}(A)\), where left invertibility means that
\(EC=I_q\) for some \(E\in M_{q,p}(A)\).  Consequently, for every
\(n\geqslant1\),
\begin{equation}\label{eq:Banach_matrix_upper_bound}
 \ltsr M_n(A)
 \leqslant
 \left\lceil\frac{\ltsr(A)-1}{n}\right\rceil+1.
\end{equation}
\end{lemma}

\begin{proof}
We prove the density assertion by induction on
\(q\).
For \(q=1\), the Monotonicity Remark shows that
\(\Lg_p(A)\) is dense whenever \(p\geqslant d\).

Now let \(q\geqslant2\), put \(k=p-q+1\), and take
\(C\in M_{p,q}(A)\).  Perturb the last \(k\) entries of the first column,
as little as desired, to a tuple \(y=(y_1,\ldots,y_k)\in\Lg_k(A)\);
the required density follows from
the Monotonicity Remark, since \(k\geqslant d\).
Choose \(c_1,\ldots,c_k\in A\) with
\(\sum_{j=1}^k c_jy_j=1\), and let \(a\) be the \((1,1)\)-entry of the
perturbed matrix.  The elementary row operation
\[
 R_1\longleftarrow R_1+
       \sum_{j=1}^k(1-a)c_jR_{q-1+j}
\]
makes that entry equal to \(1\).  Further elementary row operations
eliminate every other entry in the first column.  Thus, for some
\(L\in\GL_p(A)\), the perturbed matrix \(\widetilde C\) satisfies
\[
       L\widetilde C=
       \begin{pmatrix}1&r\\[1mm]0&D\end{pmatrix},
       \qquad D\in M_{p-1,q-1}(A).
\]
Since
\((p-1)-(q-1)+1=k\geqslant d\), the induction hypothesis permits us to
replace \(D\) by an arbitrarily close left-invertible matrix \(D'\).  If
\(ED'=I_{q-1}\), then
\[
 \begin{pmatrix}1&-rE\\[1mm]0&E\end{pmatrix}
 \begin{pmatrix}1&r\\[1mm]0&D'\end{pmatrix}
 =I_q.
\]
After \(L\) has been fixed, choose \(D'\) so close to \(D\) that
\[
       C'=L^{-1}\begin{pmatrix}1&r\\[1mm]0&D'\end{pmatrix}
\]
is as close to \(\widetilde C\) as desired.  The product displayed above,
with its first factor denoted by \(W\), shows that \(WL\) is a left inverse
for \(C'\).
Since \(\widetilde C\) was an arbitrarily small perturbation of \(C\), this
proves the required density.

Finally, as finite products of copies of \(A\), the spaces \(M_n(A)^m\)
and \(M_{mn,n}(A)\) carry equivalent standard product and matrix norms.
The map which vertically stacks the \(n\)-by-\(n\) blocks is therefore a
linear homeomorphism.  It carries \(\Lg_m(M_n(A))\) exactly onto the set of
left-invertible matrices in \(M_{mn,n}(A)\).  For
\[
       m=\left\lceil\frac{d-1}{n}\right\rceil+1
\]
we have \(mn-n+1\geqslant d\), so the density just proved gives
\eqref{eq:Banach_matrix_upper_bound}.
\end{proof}

\begin{lemma}\label{lem:rank_one_equivalence}
For every complex unital Banach algebra \(A\), the following are equivalent:
\begin{enumerate}[label=\textup{(\roman*)}]
\item \(\ltsr(A)=1\);
\item \(\rtsr(A)=1\);
\item \(\Inv(A)\) is dense in \(A\).
\end{enumerate}
\end{lemma}

\begin{proof}
This is \cite[Proposition~3.1]{RieffelStableRank}.
\end{proof}

\begin{lemma}\label{lem:quotient_stable_rank}
Let \(A\) be a unital Banach algebra and let \(J\) be a closed two-sided
ideal.  Then
\[
   \ltsr(A/J)\leqslant\ltsr(A),
   \qquad
   \rtsr(A/J)\leqslant\rtsr(A).
\]
\end{lemma}

\begin{proof}
Suppose first that \(\ltsr(A)=m<\infty\).  Given a tuple in
\((A/J)^m\), lift it to a tuple in \(A^m\) and approximate the lift by a
member of \(\Lg_m(A)\).  The quotient map is contractive and carries
left-unimodular tuples to left-unimodular tuples.  Hence \(\Lg_m(A/J)\) is
dense in \((A/J)^m\), and \(\ltsr(A/J)\leqslant m\).  The proof for the
right topological stable rank is identical.
\end{proof}

For a bounded operator \(T:Y\to Z\), write
\[
       \alpha(T)=\dim\Ker T,
       \qquad
       \beta(T)=\dim(Z/\Ran T).
\]
We call \(T\) \textit{upper semi-Fredholm} if \(\Ran T\) is closed and
\(\alpha(T)<\infty\).  We call \(T\) \textit{lower semi-Fredholm} if
\(\Ran T\) is closed and \(\beta(T)<\infty\).  Lastly, \(T\) is
\textit{Fredholm} if it is both upper and lower semi-Fredholm.  For a Fredholm
\(T\),
\[
       \ind T=\alpha(T)-\beta(T).
\]
We use the standard facts that the Fredholm index is locally constant and is
invariant under compact perturbations, and that \(I+K\) is Fredholm of index
zero whenever \(K\) is compact.

\begin{proposition}\label{prop:lower_semifred_and_index_zero}
Let \(X,Y\) be Banach spaces and let \(T\in\B(X,Y)\) be lower
semi-Fredholm, with \(\beta(T)\leqslant\alpha(T)\), where
\(\alpha(T)=\infty\) is allowed.  There is a fixed finite-dimensional
subspace \(M\subseteq Y\) such that
\[
        Y=\Ran T\oplus M,
        \qquad
        \dim M=\beta(T),
\]
and, for every \(\eta>0\), there is a finite-rank
\(F\in\B(X,Y)\) satisfying
\[
        \|F\|<\eta,
        \qquad
        \Ran F=M,
\]
such that \(T+F\) is surjective.  In particular,
\[
        \operatorname{rank}F=\beta(T),
        \qquad
        \Ran T\cap\Ran F=\{0\}.
\]
If, in addition, \(T\) is Fredholm of index zero, then \(F\) may be
chosen so that \(T+F\) is an isomorphism.
\end{proposition}

\begin{proof}
Put \(r=\beta(T)\).  Since \(\Ran T\) is closed and has codimension
\(r\), choose a fixed \(r\)-dimensional subspace \(M\subseteq Y\) such
that
\[
        Y=\Ran T\oplus M.
\]
If \(r=0\), take \(F=0\).  In the index-zero case one then also has
\(\alpha(T)=0\), so \(T\) is already an isomorphism.

Suppose that \(r>0\).  Choose an \(r\)-dimensional subspace
\(E\subseteq\Ker T\), a bounded projection \(P:X\to E\), and an
isomorphism \(R:E\to M\).  Given \(\eta>0\), choose \(\delta>0\)
so that
\[
        \delta\|R\|\|P\|<\eta
\]
and define
\[
        F=\delta\,\iota_MRP,
\]
where \(\iota_M:M\hookrightarrow Y\) is the inclusion.  Then
\(\|F\|<\eta\) and \(\Ran F=M\).

Moreover,
\[
        (T+F)(\Ker P)=\Ran T,
        \qquad
        (T+F)(E)=M.
\]
Indeed, \(T(I-P)x=Tx\) for every \(x\in X\), since
\(\Ran P\subseteq\Ker T\), while \(F\) vanishes on \(\Ker P\) and \(T\)
vanishes on \(E\).  Since \(X=E\oplus\Ker P\), it follows that \(T+F\)
is surjective.

If \(T\) is Fredholm of index zero, then
\(\dim E=r=\dim\Ker T\), so \(E=\Ker T\).  If \((T+F)x=0\), then
\[
        Tx=-Fx\in\Ran T\cap M=\{0\}.
\]
Thus \(x\in E\), while \(Fx=0\) and injectivity of \(R\) imply
\(Px=0\).  Hence \(x=0\), and \(T+F\) is an isomorphism.
\end{proof}

\begin{lemma}\label{lem:finite_codimensional_superspace}
Let \(Z\) be a Banach space, let \(M\subseteq Z\) be closed and of finite
codimension, and let \(N\) be a linear subspace with \(M\subseteq N\subseteq Z\).
Then \(N\) is closed and of finite codimension.
\end{lemma}

\begin{proof}
The quotient \(Z/M\) is finite-dimensional.  The subspace \(N/M\) is
therefore closed in \(Z/M\), and \(N\) is its inverse image under the quotient
map.  Moreover, \(Z/N\cong(Z/M)/(N/M)\) is finite-dimensional.
\end{proof}

\begin{lemma}\label{lem:finite_dimensional_lift}
Let \(H:Y\to Z\) be a bounded surjection and let \(M\subseteq Z\) be
finite-dimensional.  There is a bounded linear map \(R_M:M\to Y\) such that
\[
       H(R_Mz)=z\qquad(z\in M).
\]
\end{lemma}

\begin{proof}
Choose a basis \((m_1,\ldots,m_k)\) of \(M\) and vectors \(y_j\in Y\) with
\(Hy_j=m_j\).  Extend linearly by
\(R_M(\sum_j\lambda_jm_j)=\sum_j\lambda_jy_j\).  Every linear map with
finite-dimensional domain is bounded.
\end{proof}

Whenever a Motakis realisation is used below, we retain the rank order of
the Schauder basis furnished by the construction and relabel it as
\((e_n)_{n=1}^\infty\), with coordinate functionals
\((f_n)_{n=1}^\infty\).  For \(J\subseteq\N\), put
\[
        X_J=\overline{\operatorname{span}}\{e_n:n\in J\}.
\]
An operator \(T\in\B(X)\) is \emph{diagonal with respect to
\((e_n)\)} if
\(
       f_k(Te_n)=0
       \)
\(
    (k\neq n),
\)
or, equivalently, if \(Te_n=\lambda_ne_n\) for suitable scalars
\(\lambda_n\).  For such an operator, let
\[
        N_T=\{n\in\N:Te_n=0\}.
\]

\begin{proposition}\label{prop:diagonal_basic}
Let \(X\) have a Schauder basis \((e_n)_{n=1}^\infty\), with coordinate
functionals \((f_n)_{n=1}^\infty\), and let \(T\in\B(X)\) be diagonal.
Then
\[
        \Ker T=X_{N_T}
\]
and
\[
 \operatorname{span}\{e_n:n\notin N_T\}
 \subseteq\Ran T
 \subseteq X_{\N\setminus N_T}.
\]
Moreover, \((X/\overline{\Ran T})^*\) contains a linearly independent
family indexed by \(N_T\).
\end{proposition}

\begin{proof}
Put \(\lambda_n=f_n(Te_n)\).  Diagonality and uniqueness of basis
coefficients give
\[
        Te_n=\lambda_ne_n,
        \qquad
        N_T=\{n:\lambda_n=0\}.
\]
For \(x\in X\), the basis partial sums
\[
        P_Nx=\sum_{n=1}^N f_n(x)e_n
\]
converge to \(x\).  Since \(T\) is bounded,
\begin{equation}\label{eq:diagonal_action}
 Tx=\lim_{N\to\infty}TP_Nx
   =\lim_{N\to\infty}\sum_{n=1}^N\lambda_nf_n(x)e_n.
\end{equation}
It follows that \(X_{N_T}\subseteq\Ker T\).  Conversely, if \(Tx=0\),
uniqueness of basis coefficients in \eqref{eq:diagonal_action} gives
\(f_n(x)=0\) whenever \(n\notin N_T\), and hence \(x\in X_{N_T}\).

If \(n\notin N_T\), then
\[
        e_n=\lambda_n^{-1}Te_n,
\]
which proves the first range inclusion.  The second follows from
\eqref{eq:diagonal_action}.  Taking closures yields
\[
        \overline{\Ran T}=X_{\N\setminus N_T}.
\]

For \(n\in N_T\), the functional \(f_n\) vanishes on
\(X_{\N\setminus N_T}\), and therefore induces
\(g_n\in(X/\overline{\Ran T})^*\) such that
\[
        g_n\pi=f_n,
\]
where \(\pi:X\to X/\overline{\Ran T}\) is the quotient map.  Since the
coordinate functionals \((f_n)\) are linearly independent, so is
\((g_n)_{n\in N_T}\).
\end{proof}

\begin{lemma}\label{lem:diagonal_semifredholm}
Let \(X\) have a Schauder basis and let \(T\in\B(X)\) be diagonal.  If
\(T\) is upper or lower semi-Fredholm, then it is Fredholm and
\(\ind T=0\).
\end{lemma}

\begin{proof}
By Proposition~\ref{prop:diagonal_basic}, \(\alpha(T)<\infty\) if and
only if \(N_T\) is finite, and in that case
\[
        \alpha(T)=|N_T|.
\]
If \(T\) is lower semi-Fredholm, then \(X/\Ran T\) is finite-dimensional
and \(\Ran T\) is closed.  Hence
\(X/\Ran T=X/\overline{\Ran T}\), so the linearly independent quotient
functionals furnished by Proposition~\ref{prop:diagonal_basic} force
\(N_T\) to be finite.  If \(T\) is upper semi-Fredholm, finiteness of
\(\alpha(T)\) gives the same conclusion.

In either case \(\Ran T\) is closed, so the range inclusions in
Proposition~\ref{prop:diagonal_basic} give
\(
        \Ran T=X_{\N\setminus N_T}.
\)
The finite-rank coordinate projection
\[
        P_T=\sum_{n\in N_T}e_n\otimes f_n
\]
has kernel \(X_{\N\setminus N_T}\) and rank \(|N_T|\).  Hence
\(
 \beta(T)
 =\dim(X/\Ran T)
 =\operatorname{rank}P_T
 =|N_T|
 =\alpha(T).
\)
Thus \(T\) is Fredholm and \(\ind T=0\).
\end{proof}

\begin{lemma}\label{lem:diagonal_row}
Let \(X\) have a Schauder basis \((e_k)_{k=1}^\infty\), with coordinate
functionals \((f_k)_{k=1}^\infty\), let
\(D_1,\ldots,D_m\in\B(X)\) be diagonal, and define
\[
       R:X^m\longrightarrow X,
       \qquad
       R(x_1,\ldots,x_m)=\sum_{j=1}^mD_jx_j.
\]
If \(R\) is lower semi-Fredholm, then
\(\alpha(R)\geqslant\beta(R)\).  More precisely, if \(m=1\), then \(R\) is
Fredholm of index zero, whereas if \(m\geqslant2\), then
\(\alpha(R)=\infty\).
\end{lemma}

\begin{proof}
Since \(R\) is lower semi-Fredholm, \(\Ran R\) is closed and
\(X/\Ran R\) is finite-dimensional.
Write \(D_je_k=\lambda_{j,k}e_k\), and set
\[
       N_R=\{k\in\N:\lambda_{1,k}=\cdots=\lambda_{m,k}=0\}.
\]
If \(N_R\) were infinite, the bounded functionals \(f_k\), \(k\in N_R\),
would vanish on \(\Ran R\) and descend to infinitely many linearly
independent bounded functionals on \(X/\Ran R\), a contradiction.  Hence
\(N_R\) is finite.

For \(m=1\), apply Lemma~\ref{lem:diagonal_semifredholm}.  Suppose
\(m\geqslant2\).  For each \(k\notin N_R\), the non-zero functional
\[
       (c_1,\ldots,c_m)\longmapsto
       \sum_{j=1}^m c_j\lambda_{j,k}
\]
on \(\C^m\) has a non-zero kernel.  Choose
\(c^{(k)}=(c_{1,k},\ldots,c_{m,k})\neq0\) in this kernel and put
\[
       u_k=(c_{1,k}e_k,\ldots,c_{m,k}e_k)\in X^m.
\]
Then \(u_k\in\Ker R\).  The vectors \(u_k\), \(k\notin N_R\), are linearly
independent, so \(\Ker R\) is infinite-dimensional.
\end{proof}

\begin{lemma}\label{lem:left_invertible_column_lift}
Let \(K\) be compact metrisable, let \(X\) be an arbitrary Motakis
realisation over \(K\), put \(\mathcal A=\B(X)\),
\(\mathcal I=\K(X)\), and let
\(q:\mathcal A\to\mathcal A/\mathcal I\) be the quotient map.  Suppose that
\((\dot C_1,\ldots,\dot C_m)\in\Lg_m(\mathcal A/\mathcal I)\), and let
\(G_1,\ldots,G_m\in\mathcal A\) satisfy
\[
       q(G_j)=\dot C_j\qquad(1\leqslant j\leqslant m).
\]
Then, for every \(\eta>0\), there are finite-rank operators
\(P_1,\ldots,P_m\) with
\[
       \max_j\|P_j\|<\eta
\]
such that the column operator
\[
       G+P:X\longrightarrow X^m,
       \qquad
       (G+P)x=((G_1+P_1)x,\ldots,(G_m+P_m)x)
\]
is left-invertible.
\end{lemma}

\begin{proof}
Choose \(\dot H_1,\ldots,\dot H_m\in\mathcal A/\mathcal I\) with
\(\sum_j\dot H_j\dot C_j=1\), and use
Lemma~\ref{lem:diagonal_representative} to choose diagonal representatives
\(H_j\in\mathcal A\).  Let
\[
       H:X^m\to X,
       \qquad
       H(x_1,\ldots,x_m)=\sum_{j=1}^mH_jx_j.
\]
Then \(H\) is a diagonal row and
\[
       HG=I_X+K_0
\]
for some compact \(K_0\).  The range of \(I_X+K_0\) is closed and has finite
codimension, and it is contained in \(\Ran H\).  By
Lemma~\ref{lem:finite_codimensional_superspace}, \(H\) is lower
semi-Fredholm.  Lemma~\ref{lem:diagonal_row} and Proposition~\ref{prop:lower_semifred_and_index_zero} give a finite-rank row
\(L:X^m\to X\) such that
\(
       H'=H+L
\)
is onto.  Since \(L\) is finite-rank, \(H'G=I_X+K_1\) for a compact
\(K_1\), and this operator is Fredholm of index zero.

Apply Proposition~\ref{prop:lower_semifred_and_index_zero} to the
Fredholm operator \(U=H'G\).  It gives a fixed finite-dimensional
subspace \(M\subseteq X\) and, for every \(\delta>0\), a finite-rank
operator \(F:X\to M\), with \(\|F\|<\delta\), such that \(U+F\) is
invertible.  Since \(H'\) is onto, Lemma~\ref{lem:finite_dimensional_lift}
gives
\[
        R_M:M\to X^m,
        \qquad
        H'(R_Mz)=z\quad(z\in M).
\]
Let \(\pi_j:X^m\to X\) be the coordinate maps and put
\[
        C=\max_{1\leqslant j\leqslant m}\|\pi_jR_M\|.
\]
Choose \(F\) with \(\|F\|<\eta/(1+C)\), set \(P=R_MF\), and write
\(P_j=\pi_jP\).  Then \(P\) is finite-rank,
\[
        \max_j\|P_j\|<\eta,
\]
and
\[
 H'(G+P)=H'G+H'R_MF=U+F\in\Inv(\B(X)).
\]
Thus \((U+F)^{-1}H'\) is a left inverse of \(G+P\).
\end{proof}

\begin{proposition}\label{prop:left_stable_upper}
Let \(K\) be compact metrisable and let \(X\) be an arbitrary Motakis
realisation over \(K\).  Then
\[
       \ltsr\B(X)\leqslant\tsr C(K).
\]
\end{proposition}

\begin{proof}
There is nothing to prove if \(\tsr C(K)=\infty\).  Put
\(m=\tsr C(K)<\infty\), \(\mathcal A=\B(X)\), and
\(\mathcal I=\K(X)\).  Fix
\((T_1,\ldots,T_m)\in\mathcal A^m\) and \(\eta>0\).  Topological stable rank is
invariant under topological Banach-algebra isomorphisms and under equivalent
algebra norms.  Thus the isomorphism
\(\Psi:C(K)\to\mathcal A/\mathcal I\) transfers density
of unimodular tuples; it need not be isometric.  All quantitative
approximations below are taken in the quotient norm of
\(\mathcal A/\mathcal I\).
Let \(q:\mathcal A\to\mathcal A/\mathcal I\) be the quotient map.

Since \(\Lg_m(\mathcal A/\mathcal I)\) is dense, choose
\((\dot A_1,\ldots,\dot A_m)\in\Lg_m(\mathcal A/\mathcal I)\) such that
\[
       \max_j\|\dot A_j-q(T_j)\|_{\mathcal A/\mathcal I}<\eta/4.
\]
By the definition of the quotient norm, the \(\dot A_j\)'s have
representatives \(A_j\in\mathcal A\) satisfying
\begin{equation}\label{eq:left_stable_initial_approximation}
       \max_j\|A_j-T_j\|<\eta/2.
\end{equation}
Apply Lemma~\ref{lem:left_invertible_column_lift} to the representatives
\(A_1,\ldots,A_m\), with \(\eta/2\) in place of \(\eta\).  It gives
finite-rank operators \(P_1,\ldots,P_m\) such that
\[
       \max_j\|P_j\|<\eta/2
\]
and the column associated with
\((A_1+P_1,\ldots,A_m+P_m)\) is left-invertible.  Hence this tuple
belongs to \(\Lg_m(\mathcal A)\) and, by
\eqref{eq:left_stable_initial_approximation}, lies within \(\eta\) of
\((T_1,\ldots,T_m)\).
\end{proof}

\begin{proposition}\label{prop:right_stable_upper}
Let \(K\) be compact metrisable and let \(X\) be an arbitrary Motakis
realisation over \(K\).  Then
\[
       \rtsr\B(X)\leqslant\tsr C(K).
\]
\end{proposition}

\begin{proof}
Again suppose that \(m=\tsr C(K)<\infty\), and put
\(\mathcal A=\B(X)\), \(\mathcal I=\K(X)\).  Fix a tuple
\((T_1,\ldots,T_m)\in\mathcal A^m\) and \(\eta>0\).  As in the preceding proof,
density is transferred through \(\Psi\), while all estimates are made in the
quotient norm.  Let \(q:\mathcal A\to\mathcal A/\mathcal I\) be the quotient
map.  Choose
\((\dot A_1,\ldots,\dot A_m)\in\Rg_m(\mathcal A/\mathcal I)\) such that
\[
       \max_j\|\dot A_j-q(T_j)\|_{\mathcal A/\mathcal I}<\eta/4,
\]
and choose representatives \(A_j\in\mathcal A\) satisfying
\[
       q(A_j)=\dot A_j,
       \qquad
       \max_j\|A_j-T_j\|<\eta/2.
\]
Since
\(\mathcal A/\mathcal I\cong C(K)\) is commutative, choose
\(\dot C_j\in\mathcal A/\mathcal I\) with
\[
       \sum_{j=1}^m q(A_j)\dot C_j=1.
\]
The tuple \((\dot C_1,\ldots,\dot C_m)\) belongs to
\(\Lg_m(\mathcal A/\mathcal I)\), because
the same identity may be read with the factors reversed.  Choose diagonal
representatives \(G_j\) of the \(\dot C_j\)'s.  By
Lemma~\ref{lem:left_invertible_column_lift}, after finite-rank perturbations
which do not alter the Calkin classes, we may assume that the column
\[
       G:X\to X^m,
       \qquad
       Gx=(G_1x,\ldots,G_mx)
\]
has a left inverse \(L=(L_1,\ldots,L_m):X^m\to X\).  Thus
\(LG=I_X\).

The quotient identity gives
\[
       \mathsf R G=\sum_{j=1}^mA_jG_j=I_X+K_2,
       \qquad
       \mathsf R=(A_1,\ldots,A_m):X^m\to X,
\]
for a compact \(K_2\).  Set \(W=I_X+K_2\).  Since \(W\) is Fredholm of
index zero, Proposition~\ref{prop:lower_semifred_and_index_zero} gives
arbitrarily small finite-rank operators \(F\) for which \(W+F\) is
invertible.  Put
\[
       C=\max_{1\leqslant j\leqslant m}\|L_j\|
\]
and choose such an \(F\) with \(\|F\|<\eta/[2(1+C)]\).  Put
\(P_j=FL_j\).  Then each \(P_j\) is finite-rank,
\(\max_j\|P_j\|<\eta/2\), and
\[
       \sum_{j=1}^m(A_j+P_j)G_j
       =\mathsf R G+FLG
       =W+F\in\Inv(\mathcal A).
\]
Consequently, the column \(G(W+F)^{-1}\) is a right inverse of the row
\(\mathsf R+P=(A_1+P_1,\ldots,A_m+P_m)\).  This tuple therefore belongs to
\(\Rg_m(\mathcal A)\) and lies within \(\eta\) of the original tuple.
\end{proof}

\begin{theorem}\label{thm:motakis_exact_stable_rank}
Let \(K\) be compact metrisable and let \(X\) be any Motakis realisation
over \(K\).
Then
\[
       \ltsr\B(X)=\rtsr\B(X)=\tsr C(K).
\]
Consequently,
\[
   \ltsr\B(X)=\rtsr\B(X)=
   \begin{cases}
      \lfloor\dim K/2\rfloor+1, & \dim K<\infty,\\[1mm]
      \infty, & \dim K=\infty.
   \end{cases}
\]
\end{theorem}

\begin{proof}
The Calkin quotient and Lemma~\ref{lem:quotient_stable_rank} give
\[
       \tsr C(K)\leqslant\ltsr\B(X),
       \qquad
       \tsr C(K)\leqslant\rtsr\B(X).
\]
The reverse inequalities are Propositions~\ref{prop:left_stable_upper} and
\ref{prop:right_stable_upper}.  Rieffel's Proposition~1.7 applies to every
compact Hausdorff space; it does not assume finite covering dimension.
Equivalently, for each finite \(m\),
\[
 \Lg_m(C(K))\text{ is dense in }C(K)^m
 \quad\Longleftrightarrow\quad
 \dim K\leqslant 2m-1.
\]
It follows directly from the definition that
\[
 \tsr C(K)=
 \begin{cases}
  \lfloor\dim K/2\rfloor+1,&\dim K<\infty,\\
  \infty,&\dim K=\infty,
 \end{cases}
\]
as claimed \cite[Proposition~1.7]{RieffelStableRank}.
\end{proof}

\begin{corollary}\label{cor:matrix_amplified_stable_rank}
Let \(K\) be compact metrisable, let \(X\) be any Motakis realisation over
\(K\), and let \(n\geqslant1\).
Then
\[
   \ltsr\B(X^n)=\rtsr\B(X^n)
   =
   \begin{cases}
      \displaystyle
      \left\lceil\dfrac{\lfloor\dim K/2\rfloor}{n}\right\rceil+1,
        & \dim K<\infty,\\[3mm]
      \infty, & \dim K=\infty.
   \end{cases}
\]
\end{corollary}

\begin{proof}
Put \(\mathcal A=\B(X)\) and \(\mathcal I=\K(X)\).  Block-matrix representations yield
topological Banach-algebra isomorphisms
\[
       \B(X^n)\cong M_n(\mathcal A),
       \qquad
       \K(X^n)\cong M_n(\mathcal I),
\]
and hence
\[
       \Calk(X^n)\cong M_n(\mathcal A/\mathcal I)\cong M_n(C(K)).
\]
Since the involution on the \(C^*\)-algebra \(M_n(C(K))\) interchanges
left- and right-unimodular tuples, Lemma~\ref{lem:quotient_stable_rank}
gives
\begin{equation}\label{eq:matrix_calkin_lower_bounds}
 \tsr M_n(C(K))\leqslant\ltsr\B(X^n),
 \qquad
 \tsr M_n(C(K))\leqslant\rtsr\B(X^n).
\end{equation}

Suppose first that \(d=\dim K<\infty\), and put
\(q=\lfloor d/2\rfloor\).  For every unital \(C^*\)-algebra \(B\),
Rieffel proved
\[
       \operatorname{Bsr}(B)\leqslant\tsr(B),
\]
in Corollary~2.4~\cite[Corollary~2.4]{RieffelStableRank}.  The unnumbered
main theorem of Herman and Vaserstein gives the reverse inequality
\cite[unnumbered main theorem]{HermanVaserstein}.  Thus
\[
       \operatorname{Bsr}(B)=\tsr(B).
\]
Vaserstein's Theorem~3~\cite[Theorem~3]{Vaserstein} states that, if \(R\)
is a unital ring and
\(s=\operatorname{Bsr}(R)<\infty\), then
\[
 \operatorname{Bsr}M_n(R)
 =\left\lceil\frac{s-1}{n}\right\rceil+1.
\]
Combining these statements with Rieffel's
formula \(\tsr C(K)=q+1\) gives
\begin{equation}\label{eq:matrix_function_stable_rank}
       \tsr M_n(C(K))
       =\left\lceil\frac{q}{n}\right\rceil+1.
\end{equation}
For the arbitrary unital Banach algebra \(\mathcal A=\B(X)\),
Lemma~\ref{lem:Banach_matrix_upper_bound} gives the estimate
\[
 \ltsr M_n(\mathcal A)
 \leqslant
 \left\lceil\frac{\ltsr(\mathcal A)-1}{n}\right\rceil+1.
\]
Consequently, Theorem~\ref{thm:motakis_exact_stable_rank} yields
\[
 \ltsr\B(X^n)
 =\ltsr M_n(\mathcal A)
 \leqslant
 \left\lceil\frac{q}{n}\right\rceil+1.
\]
For clarity, the opposite-algebra identification is the anti-isomorphism
\[
 \tau:M_n(\mathcal A)\longrightarrow M_n(\mathcal A^{\mathrm{op}}),
 \qquad
 \tau\bigl((x_{ij})\bigr)=(x_{ji}),
 \qquad
 \tau(UV)=\tau(V)\tau(U).
\]
It gives
\[
 M_n(\mathcal A)^{\mathrm{op}}\cong M_n(\mathcal A^{\mathrm{op}}),
 \qquad
 \ltsr(\mathcal A^{\mathrm{op}})=\rtsr(\mathcal A).
\]
Applying the same estimate to \(\mathcal A^{\mathrm{op}}\)
gives the identical upper bound for \(\rtsr\B(X^n)\).  Together with
\eqref{eq:matrix_calkin_lower_bounds} and
\eqref{eq:matrix_function_stable_rank}, this proves the formula when
\(d<\infty\).

Now let \(\dim K=\infty\).  We show directly that
\(\tsr M_n(C(K))=\infty\), without applying a matrix formula at the value
\(\infty\).  More generally, suppose that \(A\) is a unital Banach algebra
and \(m=\ltsr M_n(A)<\infty\).  Given an arbitrary \(mn\)-tuple in \(A\),
place its entries in the first columns of \(m\) matrices in \(M_n(A)\),
putting zeros in the remaining columns.  Approximate this \(m\)-tuple of matrices
by \((Y_1,\ldots,Y_m)\in\Lg_m(M_n(A))\), and choose matrices \(B_j\) such
that
\[
       \sum_{j=1}^m B_jY_j=I_n.
\]
The \((1,1)\)-entry is the identity
\[
 \sum_{j=1}^m\sum_{r=1}^n
       (B_j)_{1r}(Y_j)_{r1}=1_A.
\]
Hence the \(mn\) entries in the first columns of the \(Y_j\)'s form a
member of \(\Lg_{mn}(A)\), arbitrarily close to the prescribed tuple.  Thus
\(\ltsr(A)\leqslant mn\).  Applying the contrapositive to \(A=C(K)\), for
which \(\tsr C(K)=\infty\), proves the claim.  The two inequalities in
\eqref{eq:matrix_calkin_lower_bounds} now force
\[
       \ltsr\B(X^n)=\rtsr\B(X^n)=\infty.
\]
\end{proof}

\begin{remark}
As a heuristic consistency check on the finite-dimensional formula, an
\(m\)-tuple in \(M_n(C(K))\) is left-unimodular exactly when its vertical
\(mn\)-by-\(n\) block matrix has rank \(n\) at every point of \(K\).  The
rank-deficient determinantal variety has real codimension
\(2(mn-n+1)\), which predicts the same dimension threshold.  This count is
not used as a proof of the lower bound.
\end{remark}

\begin{proof}[Proof of Theorem B]
The assertion for \(n=1\), including the identification with
\(\tsr C(K)\), is Theorem~\ref{thm:motakis_exact_stable_rank}.  The formula
for every \(n\geqslant1\) is
Corollary~\ref{cor:matrix_amplified_stable_rank}.  Finally,
\(\tsr C(K)=1\) if and only if \(\dim K\leqslant1\).
\end{proof}

\begin{corollary}\label{cor:prescribed_stable_rank_counterexamples}
For every \(r\in\{2,3,\ldots\}\cup\{\infty\}\), there are a Banach space
\(E_r\) and operators \(S_r,T_r\in\B(E_r)\) such that
\[
       \ltsr\B(E_r)=\rtsr\B(E_r)=r,
\]
\[
       \frac12\in\varepsilon_{\B(E_r)}(S_rT_r),
       \qquad
       \frac12\notin\varepsilon_{\B(E_r)}(T_rS_r).
\]
\end{corollary}

\begin{proof}
For \(r=2\), let \(K_2=S^4\).  For finite \(r\geqslant3\), let
\[
       K_r=S^4\,\dot\cup\,[0,1]^{4(r-1)},
\]
and let
\[
       K_\infty=S^4\,\dot\cup\,[0,1]^\N.
\]
These are compact metrisable spaces.  Define \(p_r:K_r\to S^4\) to be the
identity on the \(S^4\)-component and constant on the other component; let
\(\iota_r:S^4\to K_r\) denote the inclusion.  Pull back the functions from
Section~3:
\[
 a_r=a\circ p_r,\quad b_r=b\circ p_r,\quad
 c_r=c\circ p_r,\quad d_r=d\circ p_r,
 \quad \theta_r=\theta\circ p_r.
\]
Apply Theorem~\ref{thm:Motakis} to the entries of \(a_r,b_r\) and to
\(\theta_r\), and let \(X_r\) be the resulting Motakis realisation over
\(K_r\).  Repeating the proof of
Theorem~\ref{thm:main} gives operators \(S_r,T_r\) on
\[
       E_r=X_r^2
\]
with the asserted exponential-spectral separation.  The non-triviality of
the Calkin class is preserved: if \(c_r\) were null-homotopic, then
\(c=c_r\circ\iota_r\) would be null-homotopic, contrary to
Proposition~\ref{prop:KR}.

For finite \(r\),
\[
       \dim K_r=4(r-1),
\]
where this also holds for \(r=2\).  Hence
Corollary~\ref{cor:matrix_amplified_stable_rank}, with \(n=2\), gives
\[
 \ltsr\B(E_r)=\rtsr\B(E_r)
 =\left\lceil\frac{\lfloor4(r-1)/2\rfloor}{2}\right\rceil+1
 =r.
\]
For the infinite-rank case there is a direct cube restriction argument.
For \(N\geqslant1\), the Hilbert-cube component of \(K_\infty\) contains the
closed subspace
\[
 F_N=\{(t_j)\in[0,1]^\N:t_j=0\text{ for }j>4N\}
 \cong[0,1]^{4N}.
\]
Coordinatewise restriction, which is surjective by the Tietze extension
theorem, gives
\[
 M_2(C(K_\infty))\longrightarrow M_2(C(F_N)).
\]
Lemma~\ref{lem:quotient_stable_rank} and the finite-dimensional formula give
\[
 N+1=\tsr M_2(C(F_N))
 \leqslant\tsr M_2(C(K_\infty)).
\]
Since \(N\) is arbitrary, \(\tsr M_2(C(K_\infty))=\infty\).  Applying
Lemma~\ref{lem:quotient_stable_rank} once more to the Calkin quotient of
\(E_\infty=X_\infty^2\) gives
\[
       \ltsr\B(E_\infty)=\rtsr\B(E_\infty)=\infty.
\]
\end{proof}

\begin{proof}[Proof of Theorem A]
Use the realisation \(X\), the space \(E=X^2\), and the operators \(S,T\)
from Theorem~\ref{thm:main}.  That theorem gives the asserted
exponential-spectral separation.  Since \(\dim S^4=4\),
Corollary~\ref{cor:matrix_amplified_stable_rank}, applied with \(n=2\),
gives
\[
 \ltsr\B(E)=\rtsr\B(E)
 =\left\lceil\frac{\lfloor4/2\rfloor}{2}\right\rceil+1
 =2.
\]
\end{proof}

\section{The torsion question and first stabilisation}\label{sec:torsion}

The elementary-matrix calculation below relates a question of Klaja and
Ransford \cite[Question~4.1]{KlajaRansford} to the first stabilisation map.

For a complex unital Banach algebra \(A\), put
\[
        \IndG(A)=\Inv(A)/\Invzero(A).
\]
The identity component is a normal subgroup of the topological group
\(\Inv(A)\).  Since diagonal stabilisation sends \(\Invzero(A)\) into
\(\Invzero(M_2(A))\), it induces a group homomorphism
\[
       s_A:\IndG(A)\longrightarrow\IndG(M_2(A)),
       \qquad
       [u]\longmapsto
       \left[\begin{pmatrix}u&0\\0&1\end{pmatrix}\right].
\]

When \(1+rs\) is invertible, Weibel's Exercise~III.1.1 places
\((1+rs)(1+sr)^{-1}\) in \(E_2(R)\).  The calculation below is the
corresponding explicit Banach-algebra factorisation and gives directly the
identity-component conclusion needed here
\cite[Exercise~III.1.1]{WeibelKBook}.

\begin{lemma}\label{lem:Jacobson_stabilisation}
Let \(a,b\in A\), and suppose that
\[
       u=1-ab\in\Inv(A),
       \qquad
       v=1-ba\in\Invzero(A).
\]
Then
\[
       \begin{pmatrix}u&0\\0&1\end{pmatrix}
       \in\Invzero(M_2(A)).
\]
\end{lemma}

\begin{proof}
The following two factorisations hold in \(M_2(A)\):
\begin{align}
 \begin{pmatrix}1&0\\b&1\end{pmatrix}
 \begin{pmatrix}1&a\\0&v\end{pmatrix}
 &=\begin{pmatrix}1&a\\b&1\end{pmatrix},
 \label{eq:matrix_Jacobson_first}\\
 \begin{pmatrix}1&a\\0&1\end{pmatrix}
 \begin{pmatrix}u&0\\b&1\end{pmatrix}
 &=\begin{pmatrix}1&a\\b&1\end{pmatrix}.
 \label{eq:matrix_Jacobson_second}
\end{align}
Both unipotent matrices in these displays are exponentials.  Moreover,
\[
 \begin{pmatrix}1&a\\0&v\end{pmatrix}
 =\begin{pmatrix}1&av^{-1}\\0&1\end{pmatrix}
  \begin{pmatrix}1&0\\0&v\end{pmatrix}
 \in\Invzero(M_2(A)).
\]
Equation~\eqref{eq:matrix_Jacobson_first} therefore shows that
\(\bigl(\begin{smallmatrix}1&a\\b&1\end{smallmatrix}\bigr)\) belongs to the
identity component.  Equation~\eqref{eq:matrix_Jacobson_second} then gives
\[
       \begin{pmatrix}u&0\\b&1\end{pmatrix}
       \in\Invzero(M_2(A)).
\]
Finally,
\[
 \begin{pmatrix}1&0\\-bu^{-1}&1\end{pmatrix}
 =\exp\begin{pmatrix}0&0\\-bu^{-1}&0\end{pmatrix},
\]
and left multiplication by this exponential gives
\[
 \begin{pmatrix}1&0\\-bu^{-1}&1\end{pmatrix}
 \begin{pmatrix}u&0\\b&1\end{pmatrix}
 =\begin{pmatrix}u&0\\0&1\end{pmatrix}.
\]
\end{proof}

Define the \emph{Jacobson part} of the index group by
\[
 \Jac(A)=\{[1-ab]\in\IndG(A):
 a,b\in A,\ 1-ab\in\Inv(A),\ 1-ba\in\Invzero(A)\}.
\]
This is a distinguished subset, not necessarily a subgroup.

In this notation, the torsion problem of Klaja and Ransford is the following.

\begin{question}\label{q:torsion}
If \(\Jac(A)\) contains a non-identity element, must \(\IndG(A)\) contain a
non-trivial element of finite order?
\end{question}

\begin{proposition}\label{prop:Jacobson_kernel}
For every complex unital Banach algebra \(A\),
\[
       \Jac(A)\subseteq\ker s_A.
\]
Moreover, the following assertions are equivalent.
\begin{enumerate}[label=\textup{(\roman*)}]
\item There are \(a,b\in A\) such that
\[
       \varepsilon_A(ab)\setminus\{0\}
       \neq
       \varepsilon_A(ba)\setminus\{0\}.
\]
\item The set \(\Jac(A)\) contains a non-identity element of \(\IndG(A)\).
\end{enumerate}
\end{proposition}

\begin{proof}
The inclusion in the stabilisation kernel is
Lemma~\ref{lem:Jacobson_stabilisation}.

Suppose that \(\lambda\neq0\) witnesses non-commutativity, say
\[
       \lambda\in\varepsilon_A(ab)\setminus\varepsilon_A(ba).
\]
Then \(\lambda1-ba\in\Invzero(A)\), and hence it is invertible.  By
Jacobson's inverse formula, \(\lambda1-ab\) is invertible, but it does not
belong to \(\Invzero(A)\).  Set \(a'=\lambda^{-1}a\).  Since
\(\lambda^{-1}1_A\in\Invzero(A)\),
\[
       1-a'b=\lambda^{-1}(\lambda1-ab)
          \in\Inv(A)\setminus\Invzero(A)
\]
and
\[
       1-ba'=\lambda^{-1}(\lambda1-ba)\in\Invzero(A).
\]
Thus \([1-a'b]\) is a non-identity element of \(\Jac(A)\).  If the witness
lies in the opposite set difference, interchange \(a\) and \(b\).

Conversely, if \([1-ab]\in\Jac(A)\) is non-trivial, then
\(1-ab\notin\Invzero(A)\) and \(1-ba\in\Invzero(A)\).  Hence
\[
       1\in\varepsilon_A(ab)\setminus\varepsilon_A(ba).
\]
\end{proof}

\begin{corollary}\label{cor:injective_stabilisation_commutativity}
If \(s_A:\IndG(A)\to\IndG(M_2(A))\) is injective, then the exponential
spectrum is commutative away from zero in \(A\).
\end{corollary}

\begin{proof}
A counterexample would give a non-trivial element of \(\ker s_A\) by
Proposition~\ref{prop:Jacobson_kernel}.
\end{proof}

The following approximation argument, due in essence to Murphy
\cite{Murphy}, gives the stable-rank-one consequence without any appeal to a
stabilisation theorem.

\begin{lemma}[Murphy's approximation lemma]\label{lem:Murphy_approximation}
Let \(A\) be a complex unital Banach algebra and let \(a,b\in A\).  If either
\(a\) or \(b\) belongs to the norm closure of \(\Inv(A)\), then
\[
       \varepsilon_A(ab)\setminus\{0\}
       =
       \varepsilon_A(ba)\setminus\{0\}.
\]
\end{lemma}

\begin{proof}
Assume first that \(a\in\overline{\Inv(A)}\), and choose
\(a_n\in\Inv(A)\) with \(a_n\to a\).  Fix \(\lambda\neq0\), and put
\[
       u=\lambda1-ab,
       \qquad
       v=\lambda1-ba.
\]
By Jacobson's inverse formula, \(u\) and \(v\) are either both invertible or
both non-invertible.  In the latter case \(\lambda\) belongs to both
exponential spectra.  Suppose they are invertible.  For all sufficiently
large \(n\),
\[
       u_n=\lambda1-a_nb,
       \qquad
       v_n=\lambda1-ba_n
\]
are invertible, and the line segments from \(u\) to \(u_n\), and from \(v\)
to \(v_n\), remain in the invertible group.  Thus \(u_n\) and \(u\) lie in
the same component, and likewise \(v_n\) and \(v\).  Since \(a_n\) is
invertible,
\[
       v_n=a_n^{-1}u_na_n.
\]
Conjugation preserves the identity component, so
\[
       u\in\Invzero(A)\quad\Longleftrightarrow\quad
       v\in\Invzero(A).
\]
This proves the equality at \(\lambda\).  If
\(b\in\overline{\Inv(A)}\), use invertible approximants \(b_n\) and the
identity
\[
       \lambda1-ab_n=b_n^{-1}(\lambda1-b_na)b_n.
\]
\end{proof}

\begin{corollary}\label{cor:stable_rank_one_commutativity}
Let \(A\) be a complex unital Banach algebra.  If
\(\ltsr(A)=1\), or equivalently \(\rtsr(A)=1\), then
\[
       \varepsilon_A(ab)\setminus\{0\}
       =
       \varepsilon_A(ba)\setminus\{0\}
       \qquad(a,b\in A).
\]
\end{corollary}

\begin{proof}
By Lemma~\ref{lem:rank_one_equivalence}, \(\Inv(A)\) is dense in \(A\).
Apply Lemma~\ref{lem:Murphy_approximation}.
\end{proof}

For the original Klaja--Ransford algebra the torsion is visible:
\[
 \IndG(C(S^4,M_2(\C)))
 \cong[S^4,\GL_2(\C)]
 \cong\pi_4(\GL_2(\C))
 \cong\Z/2\Z.
\]
The first isomorphism follows from
Lemma~\ref{lem:function_components} and respects multiplication because
both sides use pointwise products.  For the second, \(\GL_2(\C)\) is
connected and, as a topological group, has trivial \(\pi_1\)-action on
\(\pi_4\); hence based and free homotopy classes coincide, even though
\(\pi_1(\GL_2(\C))\cong\Z\).  Concretely, multiplying a free homotopy
\(H_t\) pointwise by \(H_t(x_0)^{-1}\) makes it based.
Thus their class has order two.  In Theorem~\ref{thm:main}, the Calkin
quotient sends the class of \(I_E-2ST\) onto this obstruction.  This does not
determine the order of the class upstairs, since torsion in a quotient need
not lift to torsion.

\section{The block Fredholm index on essentially incomparable sums}

If \(X\) and \(Y\) are essentially incomparable, the Fredholm index of the
\((1,1)\)-entry defines a component invariant on
\(\Inv(\B(X\oplus Y))\).  We show that it has the same value on every
Jacobson pair \(I-AB,I-BA\).

Recall that \(R\in\B(X,Y)\) is \emph{inessential} if \(I_X-VR\) is Fredholm
for every \(V\in\B(Y,X)\).  We write \(\E(X,Y)\) for the inessential
operators.  They form a closed operator ideal, and an inessential perturbation
of a Fredholm operator is Fredholm with the same index
\cite{Kleinecke,AienaGonzalez}.  Banach spaces \(X\) and \(Y\) are
\emph{essentially incomparable} if
\[
       \B(X,Y)=\E(X,Y),
       \qquad
       \B(Y,X)=\E(Y,X).
\]

\begin{lemma}\label{lem:off_diagonal_inessential}
Let \(X\) and \(Y\) be essentially incomparable and put \(Z=X\oplus Y\).
For \(R_{12}\in\B(Y,X)\) and \(R_{21}\in\B(X,Y)\), the operator
\[
       R=\begin{pmatrix}0&R_{12}\\R_{21}&0\end{pmatrix}\in\B(Z)
\]
is inessential.
\end{lemma}

\begin{proof}
Let \(i_X,i_Y\) be the coordinate embeddings and \(p_X,p_Y\) the coordinate
projections.  Then
\[
       R=i_XR_{12}p_Y+i_YR_{21}p_X.
\]
Each summand is inessential by the operator-ideal property, and so is their
sum.
\end{proof}

\begin{lemma}\label{lem:block_diagonal_Fredholm}
Let \(U_X\in\B(X)\) and \(U_Y\in\B(Y)\).  The block diagonal operator
\(U_X\oplus U_Y\) on \(X\oplus Y\) is Fredholm if and only if both
\(U_X\) and \(U_Y\) are Fredholm.  In that case
\[
       \ind(U_X\oplus U_Y)=\ind U_X+\ind U_Y.
\]
\end{lemma}

\begin{proof}
One has
\[
 \Ker(U_X\oplus U_Y)=\Ker U_X\oplus\Ker U_Y,
 \qquad
 \Ran(U_X\oplus U_Y)=\Ran U_X\oplus\Ran U_Y.
\]
If the block operator is Fredholm, intersecting its closed range with the two
coordinate summands shows that \(\Ran U_X\) and \(\Ran U_Y\) are closed.  The
kernel and quotient dimensions then split as sums.  The converse and the
index formula follow from the same identities.
\end{proof}

\begin{proposition}\label{prop:phi}
Let \(X\) and \(Y\) be essentially incomparable and put \(Z=X\oplus Y\).
If
\[
        U=\begin{pmatrix}U_{11}&U_{12}\\U_{21}&U_{22}\end{pmatrix}
        \in\B(Z)
\]
is Fredholm, then \(U_{11}\) and \(U_{22}\) are Fredholm and
\[
        \ind U=\ind U_{11}+\ind U_{22}.
\]
In particular,
\[
        \varphi:\Inv(\B(Z))\longrightarrow\Z,
        \qquad
        \varphi(U)=\ind U_{11},
\]
is a continuous group homomorphism, where \(\Z\) is written additively.
\end{proposition}

\begin{proof}
By Lemma~\ref{lem:off_diagonal_inessential}, the off-diagonal part of \(U\)
is inessential.  Hence
\[
       D=\begin{pmatrix}U_{11}&0\\0&U_{22}\end{pmatrix}
\]
is Fredholm and \(\ind D=\ind U\).  Lemma~\ref{lem:block_diagonal_Fredholm}
gives the first assertion and the index formula.

The map \(U\mapsto U_{11}\) is continuous, and the Fredholm index is locally
constant, so \(\varphi\) is continuous.  If \(U,V\in\Inv(\B(Z))\), then
\[
       (UV)_{11}=U_{11}V_{11}+U_{12}V_{21}.
\]
The second summand is inessential on \(X\), while \(U_{11}V_{11}\) is
Fredholm.  Therefore
\[
 \varphi(UV)
 =\ind(U_{11}V_{11})
 =\ind U_{11}+\ind V_{11}
 =\varphi(U)+\varphi(V).
\]
\end{proof}

\begin{lemma}[Fredholm Jacobson identity]\label{lem:Fredholm_Jacobson}
Let \(A\in\B(Y,X)\) and \(B\in\B(X,Y)\).  Then \(I_X-AB\) is Fredholm if
and only if \(I_Y-BA\) is Fredholm, and in that case
\[
       \ind(I_X-AB)=\ind(I_Y-BA).
\]
\end{lemma}

\begin{proof}
Put
\[
       W=\begin{pmatrix}I_X&A\\B&I_Y\end{pmatrix}
       \in\B(X\oplus Y).
\]
The factorisations
\begin{align*}
 \begin{pmatrix}I_X&0\\B&I_Y\end{pmatrix}
 \begin{pmatrix}I_X&A\\0&I_Y-BA\end{pmatrix}
 &=W,\\
 \begin{pmatrix}I_X&A\\0&I_Y\end{pmatrix}
 \begin{pmatrix}I_X-AB&0\\B&I_Y\end{pmatrix}
 &=W
\end{align*}
have invertible first factors.  It remains only to record the indices of the
triangular second factors.  For \(T\in\B(X)\) and \(C\in\B(X,Y)\),
\[
 \Ker\begin{pmatrix}T&0\\C&I_Y\end{pmatrix}
   =\{(x,-Cx):x\in\Ker T\},
 \qquad
 \Ran\begin{pmatrix}T&0\\C&I_Y\end{pmatrix}
   =\Ran T\oplus Y.
\]
Thus this triangular operator is Fredholm exactly when \(T\) is, and has the
same index.  Similarly, for \(V\in\B(Y)\) and \(D\in\B(Y,X)\),
\[
 \Ker\begin{pmatrix}I_X&D\\0&V\end{pmatrix}
   =\{(-Dy,y):y\in\Ker V\},
 \qquad
 \Ran\begin{pmatrix}I_X&D\\0&V\end{pmatrix}
   =X\oplus\Ran V,
\]
so this operator is Fredholm exactly when \(V\) is and has index \(\ind V\).
Both factorisations of \(W\) now give the claim.
\end{proof}

\begin{proposition}\label{prop:block_index_Jacobson}
Let \(X\) and \(Y\) be essentially incomparable, put \(Z=X\oplus Y\), and
let \(A,B\in\B(Z)\).  If \(I_Z-AB\) is invertible, then \(I_Z-BA\) is
invertible and
\[
       \varphi(I_Z-AB)=\varphi(I_Z-BA).
\]
Consequently, if \(I_Z-BA\in\Invzero(\B(Z))\), then
\[
       \varphi(I_Z-AB)=0.
\]
\end{proposition}

\begin{proof}
The invertibility of \(I_Z-BA\) follows from
Lemma~\ref{lem:Jacobson_inverse}.  Write
\[
 A=\begin{pmatrix}A_{11}&A_{12}\\A_{21}&A_{22}\end{pmatrix},
 \qquad
 B=\begin{pmatrix}B_{11}&B_{12}\\B_{21}&B_{22}\end{pmatrix}.
\]
The \((1,1)\)-entry of \(I_Z-AB\) is
\[
       I_X-A_{11}B_{11}-A_{12}B_{21}.
\]
The last summand is inessential on \(X\).  Proposition~\ref{prop:phi} and
invariance of the index under inessential perturbations yield
\[
       \varphi(I_Z-AB)=\ind(I_X-A_{11}B_{11}).
\]
Similarly,
\[
       \varphi(I_Z-BA)=\ind(I_X-B_{11}A_{11}).
\]
These indices are equal by Lemma~\ref{lem:Fredholm_Jacobson}.

If \(I_Z-BA\in\Invzero(\B(Z))\), continuity of \(\varphi\) and discreteness
of \(\Z\) imply \(\varphi(I_Z-BA)=\varphi(I_Z)=0\).  The first part then
gives \(\varphi(I_Z-AB)=0\).
\end{proof}

\begin{example}\label{ex:l_infty_c0}
Gonz\'alez proved that if a Banach space \(Y\) contains no isomorphic copy of
\(\ell_1\) and \(Z\) is \(\ell_\infty\), \(H^\infty\), or \(C(K)\) for a
\(\sigma\)-Stonian compact space \(K\), then \(Y\) and \(Z\) are essentially
incomparable \cite[Theorem~1(c)]{Gonzalez}.  Moreover, \(c_0\) contains no
isomorphic copy of \(\ell_1\).  Indeed, if \(E\) is a closed subspace of
\(c_0\), the restriction map \(c_0^*=\ell_1\to E^*\) is onto, so \(E^*\) is
separable, whereas \(\ell_1^*=\ell_\infty\) is non-separable.  The criterion
therefore applies with \(Y=c_0\) and \(Z=\ell_\infty\).
On \(\ell_\infty\) and \(c_0\), respectively,
define the left and right shifts
\[
\begin{aligned}
 T_1(x_1,x_2,\ldots)&=(x_2,x_3,\ldots),
& S_1(x_1,x_2,\ldots)&=(0,x_1,x_2,\ldots),\\
 T_2(y_1,y_2,\ldots)&=(y_2,y_3,\ldots),
& S_2(y_1,y_2,\ldots)&=(0,y_1,y_2,\ldots).
\end{aligned}
\]
Let
\[
 P:\ell_\infty\to c_0,
 \qquad P(x_1,x_2,\ldots)=(x_1,0,0,\ldots),
\]
and
\[
 Q:c_0\to\ell_\infty,
 \qquad Q(y_1,y_2,\ldots)=(y_1,0,0,\ldots).
\]
Using
\[
 T_jS_j=I,
 \quad
 S_jT_j=I-P_1,
 \quad
 PQ=P_1|_{c_0},
 \quad
 QP=P_1|_{\ell_\infty},
\]
where \(P_1\) is the first-coordinate projection, one obtains directly
\[
 \begin{pmatrix}T_1&0\\P&S_2\end{pmatrix}
 \begin{pmatrix}S_1&Q\\0&T_2\end{pmatrix}
 =I_{\ell_\infty\oplus c_0}
 =
 \begin{pmatrix}S_1&Q\\0&T_2\end{pmatrix}
 \begin{pmatrix}T_1&0\\P&S_2\end{pmatrix}.
\]
Thus
\[
       M=\begin{pmatrix}T_1&0\\P&S_2\end{pmatrix}
       \in\Inv(\B(\ell_\infty\oplus c_0)).
\]
The operator \(T_1\) is onto and has one-dimensional kernel, so
\[
       \varphi(M)=\ind T_1=1.
\]
Hence \(M\notin\Invzero(\B(\ell_\infty\oplus c_0))\).  More strongly,
Proposition~\ref{prop:block_index_Jacobson} shows that there do not exist
\(A,B\in\B(\ell_\infty\oplus c_0)\) such that
\[
       M=I-AB,
       \qquad
       I-BA\in\Invzero(\B(\ell_\infty\oplus c_0)).
\]
\end{example}

\section{Open problems}

Question~\ref{q:torsion} asks whether the component detected by a Jacobson
pair must yield torsion.  For the operators in Theorem~A there is a more
specific problem.

\begin{question}\label{q:stabilisation_kernel}
Let \(E=X\oplus X\) and \(S,T\) be as in Theorem~A.  What is
\[
 \ker\bigl(s_{\B(E)}:\IndG(\B(E))\longrightarrow
                 \IndG(M_2(\B(E)))\bigr)?
\]
In particular, what is the order of \([I_E-2ST]\)?  Its image in the Calkin
index group has order two, but it is not known whether the class itself is
torsion.
\end{question}

\begin{question}\label{q:classical_space}
Can exponential-spectral commutativity fail in \(\B(X)\) when \(X\) is a
classical sequence or function space?  More generally, which structural
properties of \(X\) rule out such a failure?
\end{question}

\end{document}